\documentclass{article}
\usepackage[T1]{fontenc}
\usepackage{amsmath}
\usepackage{amsthm}
\usepackage{amsfonts}
\usepackage{microtype}
\usepackage{tikz}
\usepackage{tabularx}
\usepackage[normalem]{ulem}
\usepackage{paralist}
\usepackage{caption}
\usepackage{hyperref}
\newcommand\dv{\hspace{1pt}{:}\hspace{1pt}}

\newtheorem{theorem}{Theorem}[section]
\newtheorem{conjecture}[theorem]{Conjecture}
\newtheorem{lemma}[theorem]{Lemma}
\newtheorem{corollary}[theorem]{Corollary}
\newtheorem{proposition}[theorem]{Proposition}
\newtheorem{question}{Question}

\theoremstyle{definition}
\newtheorem{example}[theorem]{Example}

\newcommand{\defn}[1]{\emph{#1}}

\title{Relating tournaments and permutations with xrays}
\author{Matthew Davis\thanks{Muskingum University, 260 Stadium Dr., New Concord, OH 43762, USA} \and Michael W. Schroeder\thanks{Stetson University, 421 N. Woodland Dr., DeLand, FL 32723, USA}~\thanks{Corresponding author:  \tt\href{mailto:mschroeder1@stetson.edu}{mschroeder1@stetson.edu}}}
\date{}

\begin{document}
\maketitle

\begin{abstract}
In a 2005 paper, Bebeacua et al.~investigated the xrays of permutations, and conjectured a correspondence between binary xrays and score sequences of tournaments.
In 2014, Brualdi and Fritscher conjectured a possible correspondence between score sequences of $2$-tournaments and (not necessarily binary) xrays of permutations.
In this paper, we first introduce the concept of a transitive tournament decomposition of $k$-tournaments, then present a construction by which a permutation is used to build $1$- and $2$-tournaments whose score sequences agree with the xray of the permutation in the manner outlined by Bebeacua et al.~and Brualdi and Fritscher.
We close with an investigation of xrays with restricted terms, including binary xrays, and show that the recent conjectures by Bebeacua et al.~and Brualdi and Fritscher are special cases of a more general statement, which we conjecture and for which we provide supporting evidence.
\end{abstract}

\section{Introduction}
\label{sec:intro}

In this paper, we investigate the relationship between tournaments and permutations, as well as their related structures.
Let $T$ be a tournament graph on a vertex set $V$ with $n$ vertices for some positive integer $n$.
For each $v\in V$, let $s(v)$ be the outdegree of $v$; then $S(T) = \{ s(v): v\in V\}$ is the \defn{score set} of $T$ and the \defn{score sequence} of $T$, denoted $s(T)$, is the nondecreasing sequence formed from the elements of $S(T)$.
Let $Z_n = \{0,\dots,n-1\}$, and for all tournaments on $n$ vertices in this work, we assume $Z_n$ is the vertex set, unless noted otherwise.
Let $\mathcal{T}_n$ be the set of tournaments on $Z_n$ and $\mathfrak{S}_n$ be the set of score sequences of tournaments in $\mathcal{T}_n$.
Recall that Landau's theorem, a classic result in constructive combinatorics, provides a classification for when an integer sequence is a score sequence of a tournament, and thus characterizes the elements of $\mathfrak{S}_{n}$.


Now, let $\phi\in S_n$ be a permutation on $Z_n$. 
The \defn{xray} of $\phi$ is the sequence of the sums of the entries along each antidiagonal of the associated permutation matrix to $\phi$.
Define the \defn{characteristic set} of $\phi$, denoted $C(\phi)$, as $\{i+\phi(i):i\in Z_n\}$\footnote{In previous works, the characteristic set is defined as a translation of this set, but using this definition allows for the simplification of several results.}, which has $n$ elements and may be a multiset.
Let $c(\phi)$ be a nondecreasing sequence formed by the elements of $C(\phi)$; call this the \defn{characteristic sequence} of $\phi$.
The characteristic sequences of permutations on $Z_n$ are in one-to-one correspondence with the xrays of permutations on $Z_n$; if $C(\phi)$ is simple (not a multiset), that is, if $c(\phi)$ is increasing, then the xray of $\phi$ is \defn{binary}.
Let $\mathfrak{C}_n$ be the set of characteristic sequences of permutations on $Z_n$, and let $\overline{\mathfrak{C}}_n$ be the subset of $\mathfrak{C}_n$ which contains increasing sequences.

A set of necessary numerical conditions exists for an increasing integer sequence to be a possible characteristic sequence of a permutation on $Z_n$. 
These conditions are essentially equivalent to those in Landau's Theorem, so that the number of integer sequences which satisfy these conditions equals $|\mathfrak{S}_n|$. 
This led Bebeacua et al.~to the following theorem:

\begin{theorem}
\label{thm:bebeacua}
For all positive integers $n$, $|\overline{\mathfrak{C}}_n| \leq |\mathfrak{S}_n|$.
\end{theorem}

With numerical supporting evidence (up through at least $n=12$), this allowed the following conjecture:
\begin{conjecture}
\label{conj:bebeacua}
For all positive integers $n$, $|\overline{\mathfrak{C}}_n| = |\mathfrak{S}_n|$.
\end{conjecture}

Since their paper in 2005, there has been modest headway in finding a proof of Conjecture \ref{conj:bebeacua}.
In 2014, Brualdi and Fritscher~\cite{brualdi} explored several aspects of xrays of permutations, including a correlation with score sequences of $2$-tournaments.
For a positive integer $k$, let $\mathcal{T}_n^k$ denote the set of $k$-tournaments on $Z_n$ and $\mathfrak{S}_n^k$ denote the set of score sequences of $k$-tournaments in $\mathcal{T}_n^k$.
In their paper, Brualdi and Fritscher prove the following theorem.

\begin{theorem}
\label{thm:brualdi}
For all positive integers $n$, $\mathfrak{C}_n \subseteq \mathfrak{S}^2_n$.
\end{theorem}

Given its similarity to the correspondence presented by Bebeacua et al., this led Brualdi and Fritscher to conjecture that $\mathfrak{C}_n = \mathfrak{S}^2_n$ as well.
However, a 2008 result by Nordh~\cite{nordh1} shows this conjecture to be false by finding a parallel between $\mathfrak{C}_n$ and perfect extremal Skolem sets.
Furthermore, Yu~\cite{Yu} showed that the problem of producing a permutation which yields a potential xray is NP-complete, so finding a constructive argument may be difficult.
However, the complexity of certain variations of this problem are not known.
It was noted by Brunetti et al.~\cite{Brunetti} that it is not known whether producing a permutation which yields a potential binary xray is NP-complete.
Del Lungo~\cite{Lungo} showed that a polynomial-time algorithm exists if one requires the associated permutation matrix be wrapped around a cylinder.
This parallels a similar process given by Marshall Hall Jr.~\cite{HallJr} for solving a related problem involving Abelian groups.

Theorems \ref{thm:bebeacua} and \ref{thm:brualdi} involve showing correspondences between the elements of $\overline{\mathfrak{C}}_n$ and $\mathfrak{S}_n$, and between the elements of $\mathfrak{C}_n$ and $\mathfrak{S}^2_n$, respectively.
None of the arguments focus on finding relationships between the elements in $S_n$ and $\mathcal{T}_n$ or $\mathcal{T}_n^2$.
We first prove there is a correspondence between permutations and $2$-tournaments that respects the correspondence between xrays and characteristic sequences with score sequences.

\begin{theorem}
\label{thm:new_thm}
Let $n$ be a positive integer.
There exists a function $\gamma: S_n \to \mathcal{T}_n^2$ such that $c = s\circ \gamma$.
\end{theorem}

We then look at several consequences of this construction.
In Section \ref{sec:prelim}, we present necessary definitions and notation for our objects of study.
In Section \ref{sec:tournaments}, we give a construction for building a $2$-tournament from a permutation which demonstrates Theorem \ref{thm:new_thm}.
We introduce the concept of decomposing $k$-tournaments into the union of transitive tournaments, and show that the existence of a decomposition of a $2$-tournament into transitive tournaments is required for a correspondence between 
elements of $S_n$ and $\mathcal{T}_n^2$, and similarly between $\mathfrak{C}_n$ and $\mathfrak{S}_n^2$. 
We also consider the implications of this correspondence on the relationship between $\overline{\mathfrak{C}}_n$ and $\mathfrak{S}_n$, as highlighted in Conjecture \ref{conj:bebeacua}.

In Section \ref{sec:restricted}, we investigate score sequences of $2$-tournaments and characteristic sequences of permutations with restrictions on the multiplicity of their terms.
We conclude showing that  Conjecture \ref{conj:bebeacua} and another conjecture by Bebeacua et al.~are special cases of a more comprehensive statement.
We formulate this more general conjecture and provide supporting evidence for it.

\section{Preliminaries}
\label{sec:prelim}

In this section, we review the results and notation related to tournaments, permutations, and xrays.
Unless stated otherwise, a permutation in $S_n$ is assumed to act on $Z_n$.

\subsection{Tournaments}

Let $n$ be a positive integer, and let $T$ be a digraph on $n$ vertices.
Then $T$ is a \defn{tournament} if for each distinct pair $x,y\in V(T)$, either the directed edge $(x,y)\in E(T)$ or $(y,x)\in E(T)$, but not both.
This language aligns with recording the results of a round-robin tournament -- a tournament in which each pair of players compete and each contest results in a winner and loser; there are no ties.
We can produce a tournament by including the edge $(x,y)\in E(T)$ if and only if $x$ beats $y$ in their game.
We use the notation $x >_T y$ to mean $(x,y)\in E(T)$, or equivalently, $x$ beats $y$ in $T$.

Let $A_T$ be the adjacency matrix of $T$, where the rows and columns of $A_T$ are indexed by $Z_n$.
Observe that $A_T+A_T' = J_n-I_n$, where $A_T'$ is the transpose of $A_T$, $I_n$ is the identity matrix of order $n$, and $J_n$ is the $n\times n$ matrix of all 1s.
\begin{figure}
\centering
\begin{tabular}{@{}cc@{\qquad\qquad}cc@{}}
\begin{tikzpicture}[line width=1pt,scale=1.6]
\coordinate (0) at (0,0);
\coordinate (3) at (0,1);
\coordinate (1) at (1,0);
\coordinate (2) at (1,1);
\node[circle,draw,minimum width=13pt] (a0) at (0) {};
\node[circle,draw,minimum width=13pt] (a1) at (1) {};
\node[circle,draw,minimum width=13pt] (a2) at (2) {};
\node[circle,draw,minimum width=13pt] (a3) at (3) {};
\node at (0) {0};
\node at (1) {1};
\node at (2) {2};
\node at (3) {3};
\draw[->] (a3) -- (a0);
\draw[->] (a3) -- (a2);
\draw[->] (a3) -- (a1);
\draw[->] (a0) -- (a2);
\draw[->] (a0) -- (a1);
\draw[->] (a2) -- (a1);
\end{tikzpicture}
&
\begin{tikzpicture}
\node at (0,0) {\begin{tabular}{|c|c|c|c|}
\hline
0 & 1 & 1 & 0\\\hline
0 & 0 & 0 & 0\\\hline
0 & 1 & 0 & 0\\\hline
1 & 1 & 1 & 0\\\hline
\end{tabular}};
\end{tikzpicture}
&
\begin{tikzpicture}[line width=1pt,scale=1.6]
\coordinate (0) at (0,0);
\coordinate (3) at (0,1);
\coordinate (1) at (1,0);
\coordinate (2) at (1,1);
\node[circle,draw,minimum width=13pt] (a0) at (0) {};
\node[circle,draw,minimum width=13pt] (a1) at (1) {};
\node[circle,draw,minimum width=13pt] (a2) at (2) {};
\node[circle,draw,minimum width=13pt] (a3) at (3) {};
\node at (0) {0};
\node at (1) {1};
\node at (2) {2};
\node at (3) {3};
\draw[->] (a0) -- (a1);
\draw[->] (a0) -- (a3);
\draw[->] (a1) -- (a2);
\draw[->] (a2) -- (a0);
\draw[->] (a3) -- (a1);
\draw[->] (a3) -- (a2);
\end{tikzpicture}
&
\begin{tikzpicture}
\node at (0,0) {\begin{tabular}{|c|c|c|c|}
\hline
0 & 1 & 0 & 1\\\hline
0 & 0 & 1 & 0\\\hline
1 & 0 & 0 & 0\\\hline
0 & 1 & 1 & 0\\\hline
\end{tabular}};
\end{tikzpicture}
\\
\multicolumn{2}{c}{(a)} &
\multicolumn{2}{c}{(b)} 
\end{tabular}
\caption[fragile]{\begin{tabularx}{.85\linewidth}[t]{@{}l@{\ }X@{}}
(a) & A transitive tournament and its adjacency matrix. \\
(b) & A non-transitive tournament and its adjacency matrix.
\end{tabularx}
}
\label{fig:tourn}
\end{figure}
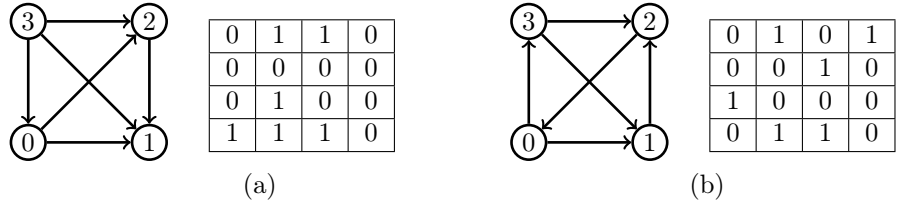
See Figure \ref{fig:tourn} for two tournaments in $\mathcal{T}_4$, as well as their adjacency matrices.

Now, let $k$ and $n$ be positive integers.
Let $T$ be a complete digraph on $n$ vertices with an associated weight function $\omega_T:E(T)\to\mathbb{Z}^{\geq 0}$.
Then $T$ is a \defn{$k$-tournament} if $\omega(x,y)+\omega(y,x) = k$ for each pair $x,y\in V(T)$.
This is analogous to a method for recording the results of a round-robin tournament in which each pair of teams play $k$ games with no ties, and $\omega_T(x,y)$ is the number of games in which $x$ beats $y$ in the tournament.
Note that a $1$-tournament is equivalent to a tournament.
Similarly, we assume that $V(T)=Z_n$ unless stated otherwise, and let $A_T$ be the associated adjacency matrix for $T$ whose rows and columns are indexed by $Z_n$; that is $A_T(i,j) = \omega_T(i,j)$ for each distinct pair $i,j\in Z_n$.
Note that $A_T + A_T' = k(J_n-I_n)$.
An example of a $2$-tournament in $\mathcal{T}_4^2$ and its adjacency matrix is given in Figure \ref{fig:ktourn}(a).
\begin{figure}
\centering
\begin{tabular}{@{}ccc@{}}
\begin{tikzpicture}[line width=1pt,scale=1.2]
\coordinate (0) at (0,0);
\coordinate (3) at (0,1.3);
\coordinate (1) at (1,0);
\coordinate (2) at (1,1.3);
\node[circle,draw,minimum width=13pt] (a0) at (0) {};
\node[circle,draw,minimum width=13pt] (a1) at (1) {};
\node[circle,draw,minimum width=13pt] (a2) at (2) {};
\node[circle,draw,minimum width=13pt] (a3) at (3) {};
\node at (0) {0};
\node at (1) {1};
\node at (2) {2};
\node at (3) {3};
\draw[->,lightgray] (a1) to[bend right=10] (a0);
\draw[->,lightgray] (a1) to[bend  left=10] (a2);
\draw[->,lightgray] (a1) to[bend right=10] (a3);
\draw[->,lightgray] (a0) to[bend  left=10] (a2);
\draw[->,lightgray] (a0) to[bend right=10] (a3);
\draw[->,lightgray] (a3) to[bend right=10] (a2);
\draw[->] (a2) to[bend  left=10] (a0);
\draw[->] (a2) to[bend  left=10] (a1);
\draw[->] (a2) to[bend right=10] (a3);
\draw[->] (a0) to[bend right=10] (a1);
\draw[->] (a0) to[bend  left=10] (a3);
\draw[->] (a1) to[bend  left=10] (a3);
\end{tikzpicture}
\begin{tikzpicture}
\node at (0,0) {\begin{tabular}{|@{\ }c@{\ }|@{\ }c@{\ }|@{\ }c@{\ }|@{\ }c@{\ }|}
\hline
0 & 1 & 1 & 2\\\hline
1 & 0 & 1 & 2\\\hline
1 & 1 & 0 & 1\\\hline
0 & 0 & 1 & 0\\\hline
\end{tabular}};
\end{tikzpicture}
&
\begin{tikzpicture}[line width=1pt,scale=1.2]
\coordinate (0) at (0,0);
\coordinate (3) at (0,1.3);
\coordinate (1) at (1,0);
\coordinate (2) at (1,1.3);
\node[circle,draw,minimum width=13pt] (a0) at (0) {};
\node[circle,draw,minimum width=13pt] (a1) at (1) {};
\node[circle,draw,minimum width=13pt] (a2) at (2) {};
\node[circle,draw,minimum width=13pt] (a3) at (3) {};
\node at (0) {0};
\node at (1) {1};
\node at (2) {2};
\node at (3) {3};
\draw[->] (a2) to[bend left=10] (a0);
\draw[->] (a2) to[bend  left=10] (a1);
\draw[->] (a2) to[bend right=10] (a3);
\draw[->] (a0) to[bend right=10] (a1);
\draw[->] (a0) to[bend  left=10] (a3);
\draw[->] (a1) to[bend left=10] (a3);
\end{tikzpicture}
\begin{tikzpicture}
\node at (0,0) {\begin{tabular}{|@{\ }c@{\ }|@{\ }c@{\ }|@{\ }c@{\ }|@{\ }c@{\ }|}
\hline
0 & 1 & 0 & 1\\\hline
0 & 0 & 0 & 1\\\hline
1 & 1 & 0 & 1\\\hline
0 & 0 & 0 & 0\\\hline
\end{tabular}};
\end{tikzpicture}
&
\begin{tikzpicture}[line width=1pt,scale=1.2]
\coordinate (0) at (0,0);
\coordinate (3) at (0,1.3);
\coordinate (1) at (1,0);
\coordinate (2) at (1,1.3);
\node[circle,draw,minimum width=13pt] (a0) at (0) {};
\node[circle,draw,minimum width=13pt] (a1) at (1) {};
\node[circle,draw,minimum width=13pt] (a2) at (2) {};
\node[circle,draw,minimum width=13pt] (a3) at (3) {};
\node at (0) {0};
\node at (1) {1};
\node at (2) {2};
\node at (3) {3};
\draw[->,lightgray] (a1) to[bend right=10] (a0);
\draw[->,lightgray] (a1) to[bend  left=10] (a2);
\draw[->,lightgray] (a1) to[bend right=10] (a3);
\draw[->,lightgray] (a0) to[bend  left=10] (a2);
\draw[->,lightgray] (a0) to[bend right=10] (a3);
\draw[->,lightgray] (a3) to[bend right=10] (a2);
\end{tikzpicture}
\begin{tikzpicture}
\node at (0,0) {\begin{tabular}{|@{\ }c@{\ }|@{\ }c@{\ }|@{\ }c@{\ }|@{\ }c@{\ }|}
\hline
0 & 0 & 1 & 1\\\hline
1 & 0 & 1 & 1\\\hline
0 & 0 & 0 & 0\\\hline
0 & 0 & 1 & 0\\\hline
\end{tabular}};
\end{tikzpicture}
\\
(a) & (b) & (c)\end{tabular}
\caption{A $2$-tournament with its adjacency matrix (a), and two transitive tournaments with their adjacency matrices (b and c) which comprise a transitive tournament decomposition of the $2$-tournament.}
\label{fig:ktourn}
\end{figure}
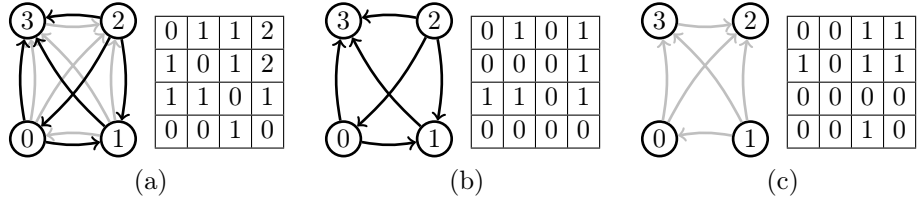
We define the score set of $T$, denoted as $S(T)$, as the (multi)set of row sums of $A_T$; equivalently $S(T)$ is the (multi)set of outdegrees for the vertices in $T$.
Finally, we define the score sequence of $T$, denoted as $s(T)$, as the nondecreasing sequence formed from the elements of $S(T)$.
The score sequences of the tournaments given in Figures \ref{fig:tourn}(a), \ref{fig:tourn}(b), and \ref{fig:ktourn}(a), are  $(0,1,2,3)$, $(1,1,2,2)$, and $(1,3,4,4)$, respectively.

The following result by Landau characterizes the sequences in $\mathfrak{S}_n$.

\begin{theorem}
[Theorem 2.2.2, \cite{brualdicmc}]
\label{thm:landau}
Let $n$ be a positive integer and $s=(s_0,s_1,\dots,s_{n-1})$ be a sequence of length $n$ of nondecreasing integers.
Then $s\in \mathfrak{S}_n$ if and only if 
\begin{equation}
\sum_{i=1}^t s_i \geq \binom t2
\label{eq:landaucondition}
\end{equation}
for each $t\in\{1,\dots,n\}$, and achieves equality when $t=n$.
\end{theorem}

The previous theorem can be generalized to characterize the elements of $\mathfrak{S}_n^k$.

\begin{theorem}[Theorem 2.2.4, \cite{brualdicmc}]
\label{thm:klandau}
Let $n$ and $k$ be a positive integers and $s=(s_0,s_1,\dots,s_{n-1})$ be a sequence of length $n$ of nondecreasing integers.
Then $s\in\mathfrak{S}_n^k$ if and only if 
\begin{equation}
\sum_{i=1}^t s_i \geq k\cdot\binom t2
\label{eq:klandaucondition}
\end{equation}
for each $t\in\{1,\dots,n\}$, and achieves equality when $t=n$.
\end{theorem}

Let $n$ be a positive integer and $T\in \mathcal{T}_n$.
We say $T$ is \defn{transitive} if for all $x,y,z\in Z_n$, if $x >_T y$ and $y>_T z$, then $x>_T z$ as well.
Equivalently, $T$ is transitive if and only if $S(T)=Z_n$, and $T$ is transitive if and only if there exists an ordering $s_0,\dots,s_{n-1}$ of the elements of $Z_n$ such that $s_i=s_T(i)$ for all $i\in Z_n$. 
We denote such a tournament as $T(s_0\cdots s_{n-1})$, and we denote the adjacency matrix of such a tournament as $A(s_0\cdots s_{n-1})$.
Equivalently, if we let $\phi$ represent the permutation $s_0\cdots s_{n-1}$, in single word notation, then we abbreviate these objects as $T(\phi)$ and $A(\phi)$, respectively.
For example, let $T$ be the tournament given in Figure \ref{fig:tourn}(a).
Then $s_T(0)=2$, $s_T(1)=0$, $s_T(2)=1$, and $s_T(3)=3$, so $S(T)=Z_4$.
Then $T=T(2013)$.
Observe that $A(2013)$ is also given in Figure \ref{fig:tourn}(a).

Let $k$ and $n$ be positive integers and $T\in\mathcal{T}_n^k$.
Suppose $T_1,\dots,T_k\in \mathcal{T}_n$ such that $A_{T} = A_{T_1}+\cdots+A_{T_k}$.
Then we say that $T$ has a \defn{tournament decomposition}. 
In this case, we will abuse notation slightly to write $T = T_1 + T_2 + \ldots + T_k$.
Every $k$-tournament has a tournament decomposition, some $k$-tournaments have unique decompositions, and others have many distinct decompositions.
For example, for all $T\in\mathcal{T}_n$, $kT = T + \cdots + T$ has a unique tournament decomposition.
However, for all tournaments $T\in\mathcal{T}_n$, the complete digraph $D_n$ on $n$ vertices may be decomposed as $T + \overline{T}$, where $\overline{T}$ is the complement of $T$ in $D_n$.

If $T\in\mathcal{T}^k_n$ and $T_1,T_2,\dots,T_k\in \mathcal{T}_n$ such that $T=T_1+\cdots+T_k$, and $T_i$ is transitive for each $i\in\{1,\dots,k\}$, we say $T$ has a \defn{transitive tournament decomposition}, abbreviated \defn{TTD}.
For example, let $T\in \mathcal{T}_4^2$ as given in Figure \ref{fig:ktourn}(a). Then $T = T(2130) + T(2301)$, and hence $T$ has a TTD.
The adjacency matrices for the transitive tournaments in this decomposition are given in Figures \ref{fig:ktourn}(b) and \ref{fig:ktourn}(c).
Also note that $S(T)$ is easily produced by entry-wise addition of the integers in the sequences $2130$ and $2301$, giving that $S(T) = \{4,4,3,1\}$.
Finally, observe there are tournaments in $\mathcal{T}_n^k$ with no TTD. 
For example, if $T\in \mathcal{T}_n$ is not transitive, then $kT \in \mathcal{T}_n^k$ and $kT$ has a unique tournament decomposition into copies of $T$, each of which is not transitive.

\subsection{Permutations}

Let $n$ be a positive integer and $\phi\in S_n$.
Denote the permutation matrix associated with $\phi$ as $P_\phi$.
The \defn{xray} of $\phi$, denoted as $x(\phi)$, is the sequence $(x_0,\dots,x_{2n-2})$ of length $2n-1$ such that for each $i\in Z_{2n-1}$, $x_i(\phi)$ is the sum of the terms in the $i$th antidiagonal of $P_\phi$; that is,
\begin{equation*}
x_i(\phi) = P_\phi(0,i) + P_\phi(1,i-1) + \cdots + P_\phi(i,0),
\end{equation*}
with the convention that $P_\phi(x,y)=0$ if $x\geq n$ or $y\geq n$.\footnote{In previous works, the indexing of the xray begins at 1, where we use 0 in this paper.
This adjustment for the indexing simplifies the statements and arguments later in the paper.}
If $x_i(\phi)\in\{0,1\}$ for each $i\in Z_{2n-1}$, we say $\phi$ has a \defn{binary xray}.
For example, the permutation $120\in S_3$ has the binary xray $x(120) = (0,1,1,1,0)$, and $021\in S_3$ has the xray $x(021)=(1,0,0,2,0)$, which is not binary.

Let $(i_1,\dots,i_t)$ be the increasing sequence of indices such that $x_{i_j}(\phi)$ is positive, for each $j\in\{1,\dots,t\}$.
Define the \defn{characteristic set} of $\phi$, denoted $C(\phi)$, as the (multi)set given by the union of $x_{i_j}$ copies of $i_j$, for each $j\in\{1,\dots,t\}$, or equivalently, $C(\phi)=\{i+\phi(i):i\in Z_n\}$.
That is, $C(\phi)$ is the (multi)set of indices for the antidiagonals for the 1s in $P_\phi$.
Let the \defn{characteristic sequence} of $\phi$, denoted $c(\phi)$, be the nondecreasing sequence of the terms in $c(\phi)$.
For example, $C(120)=\{1,3,2\}$, $c(120)=(1,2,3)$, $C(021)=\{0,3,3\}$ and $c(021)=(0,3,3)$.
Note that $x(\phi)$ is binary if and only if $C(\phi)$ is simple (not a multiset). 
For a given permutation $\phi$, we have that $x(\phi)$, $c(\phi)$ and $C(\phi)$ all contain structurally equivalent information about $\phi$.
Bebeacua et al.~use $x(\phi)$ for a majority of their arguments, while Brualdi and Fritscher~\cite{brualdi} use $c(\phi)$ for most of their work.
In our arguments, we primarily use $C(\phi)$ and $c(\phi)$, as they align better with score sequences of tournaments, in practice.

As established by previous works~\cite{bebeacua,brualdi}, the entries in a characteristic sequence satisfy a family of necessary conditions.

\begin{lemma}
\label{lem:xrayconditions}
Let $n$ be a positive integer and $\phi\in S_n$.
Let $c(\phi) = (c_0,\dots,c_{n-1})$.
If $1\leq t \leq n$,
\begin{equation}
\sum_{i=0}^{t-1} c_i \geq 2\binom t2,
\label{eq:xraycondition}
\end{equation}
with equality achieved when $t=n$.
\end{lemma}

\begin{proof}
Suppose $1\leq t \leq n$.
For each $i\in Z_n$, define $x_i$ and $y_i$ so that $x_i+y_i = c_i$ and $\phi(x_i)=y_i$.
Then $\{x_i:i\in Z_n\}=\{y_i:i\in Z_n\} = Z_n$.
So
\begin{align*}
\sum_{i=0}^{t-1} c_i 
&= \sum_{i=0}^{t-1} (x_i+y_i)
= \left(\sum_{i=0}^{t-1} x_i\right)+\left(\sum_{i=0}^{t-1} y_i\right)
\geq \binom{t}2+\binom t2
= 2\binom t2,
\end{align*}
and we have equality when $t=n$.
\end{proof}

For each positive integer $n$, define $\overline{\mathfrak{S}_n^2}$ as the subset of $\mathfrak{S}_n^2$ which contains increasing sequences.
Let $\alpha:\overline{\mathfrak{S}_n^2}\to \mathfrak{S}_n$ be defined such that $\alpha(s_0,\dots,s_{n-1}) = (s'_0,\dots,s'_{n-1})$, where $s'_i = s_i-i$ for each $i\in Z_n$.
Bebeacua et al.~showed that $\alpha$ is a bijection, and hence $|\overline{\mathfrak{S}_n^2}|=|\mathfrak{S}_n|$.
Observe that the inequalities given in \eqref{eq:xraycondition} are identical to those found in \eqref{eq:klandaucondition} presented in Theorem \ref{thm:klandau} with $k=2$.
So $\overline{\mathfrak{C}}_n \subseteq \overline{\mathfrak{S}_n^2}$, which proves Theorem \ref{thm:bebeacua} and led to Conjecture \ref{conj:bebeacua}.


Finally, we state without proof the following property of xrays when equality in~\eqref{eq:xraycondition} is met for some value $t < n$ by the characteristic sequence of a permutation.

\begin{lemma}\label{lem:PermReduc}
Let $n$ be a positive integer, $\phi\in S_n$, and $c(\phi)=(c_0,\dots,c_{n-1})$.
Suppose that, for an integer $t$ with $1\leq t \leq n-1$,
\begin{equation*}
\sum_{i=0}^{t-1} c_i = 2 \binom t2.
\end{equation*}
Then $\phi$ is a reducible permutation, and $\phi$ may be written as $\phi = \phi_1\phi_2$, where $\phi_1\in S_t$ and $\phi_2\in S_{n-t}$ acting on $Z_n\backslash Z_t$.
Equivalently, $P_\phi=P_{\phi_1}\oplus P_{\phi_2}$, which is the diagonal block matrix with $P_{\phi_1}$ and $P_{\phi_2}$ appearing on the main diagonal.
\end{lemma}

\section{Permutations and transitive tournaments}
\label{sec:tournaments}

In this section, we explore a more structural relationship between permutations and tournaments.
We first highlight the correspondence between $S_n$ and $\mathcal{T}_n^2$
by giving a method that, for a permutation $\phi\in S_n$, produces a $2$-tournament $T\in \mathcal{T}_n^2$ such that $c(\phi) = s(T)$.
We show that TTDs of $2$-tournaments are necessary for a score sequence to give rise to an xray of a permutation.
We conclude this section with an exposition on a generalization of binary xrays.
Note that, in what follows, we use parenthetical notation to abbreviate sequences.
For example, $(5:1)(3:2)(5:2)(4:3)$ represents the sequence $5,3,3,5,5,4,4,4$.

\subsection[Permutations and 2-tournaments]{Permutations and $2$-tournaments}

Let $n$ be a positive integer 
and $\phi\in S_n$.
Recall that $T(\phi)\in \mathcal{T}_n$ is the tournament for which $i$ has $\phi(i)$ wins for each $i\in Z_n$, or equivalently the team with $i$ wins is $\phi^{-1}(i)$, for each $i\in Z_n$.
The map $\phi \leftrightarrow T(\phi)$ is a natural bijection between permutations and transitive tournaments, and we denote $\phi_T$ as the permutation corresponding to $T$ in this way.
Note that $\phi_T(i)=s_T(i)$ for each $i\in Z_n$.
With this in mind, we can also find a $2$-tournament that naturally corresponds to a permutation with respect to its xray.

\begin{lemma}
\label{lem:2tourn}
Let $n$ be a positive integer and $\phi\in S_n$.
Then there exists a $2$-tournament $T\in\mathcal{T}_n^2$ with a TTD such that $s(T)=c(\phi)$.
\end{lemma}

It should be noted that this lemma follows directly from Theorem \ref{thm:klandau}.
However, the following proof gives a more structural relationship between elements of $\mathcal{T}_n^2$ and $S_n$.
Furthermore, this lemma directly implies Theorem \ref{thm:new_thm}.

\begin{proof}
Let $\epsilon\in S_n$ be the identity permutation, and define $T=T(\epsilon)+T(\phi)$.
Then $s_{T(\epsilon)}(i) = \epsilon(i) = i$ and $s_{T(\phi)}(i) = \phi(i)$ for each $i\in Z_n$, so $s_T(i) = i+\phi(i)$.
Hence, as multisets,
$S(T) = \{ i+\phi(i): i\in Z_n\} = C(\phi)$, and so $s(T)=c(\phi)$.
\end{proof}

For $\phi\in S_n$, we denote  $T(\epsilon)+T(\phi)$ as $T^2(\phi)$.
So $c(\phi) = s(T^2(\phi))$.
For example, let $\phi = 2310$.
Then $C(\phi)=\{2,4,3,3\}$.
So, as outlined in Lemma \ref{lem:2tourn}, $T^2(\phi)=T(0123)+T(2310)$, and $S(T^2(\phi)) = \{2,4,3,3\}$ as well.
The tournament $T^2(\phi)$, its decomposition, and associated adjacency matrices are given in Figure \ref{fig:phi2tourn}.
Note that, in general, any $2$ in $A_{T^2(\phi)}$ must appear below the main diagonal.
\begin{figure}
\centering
\begin{tabular}{@{}ccc@{}}
\begin{tikzpicture}[line width=1pt,scale=1.2]
\coordinate (0) at (0,0);
\coordinate (3) at (0,1.3);
\coordinate (1) at (1,0);
\coordinate (2) at (1,1.3);
\node[circle,draw,minimum width=13pt] (a0) at (0) {};
\node[circle,draw,minimum width=13pt] (a1) at (1) {};
\node[circle,draw,minimum width=13pt] (a2) at (2) {};
\node[circle,draw,minimum width=13pt] (a3) at (3) {};
\node at (0) {0};
\node at (1) {1};
\node at (2) {2};
\node at (3) {3};
\draw[->,lightgray] (a1) to[bend right=10] (a0);
\draw[->,lightgray] (a1) to[bend left=10] (a2);
\draw[->,lightgray] (a1) to[bend right=10] (a3);
\draw[->,lightgray] (a0) to[bend left=10] (a2);
\draw[->,lightgray] (a0) to[bend right=10] (a3);
\draw[->,lightgray] (a2) to[bend left=10] (a3);
\end{tikzpicture}
\begin{tikzpicture}
\node at (0,0) {\begin{tabular}{|@{\ }c@{\ }|@{\ }c@{\ }|@{\ }c@{\ }|@{\ }c@{\ }|}
\hline
0 & 0 & 1 & 1\\\hline
1 & 0 & 1 & 1\\\hline
0 & 0 & 0 & 1\\\hline
0 & 0 & 0 & 0\\\hline
\end{tabular}};
\end{tikzpicture}
&
\begin{tikzpicture}[line width=1pt,scale=1.2]
\coordinate (0) at (0,0);
\coordinate (3) at (0,1.3);
\coordinate (1) at (1,0);
\coordinate (2) at (1,1.3);
\node[circle,draw,minimum width=13pt] (a0) at (0) {};
\node[circle,draw,minimum width=13pt] (a1) at (1) {};
\node[circle,draw,minimum width=13pt] (a2) at (2) {};
\node[circle,draw,minimum width=13pt] (a3) at (3) {};
\node at (0) {0};
\node at (1) {1};
\node at (2) {2};
\node at (3) {3};
\draw[->] (a3) to[bend left=10] (a2);
\draw[->] (a3) to[bend right=10] (a1);
\draw[->] (a3) to[bend right=10] (a0);
\draw[->] (a2) to[bend left=10] (a1);
\draw[->] (a2) to[bend left=10] (a0);
\draw[->] (a1) to[bend left=10] (a0);
\end{tikzpicture}
\begin{tikzpicture}
\node at (0,0) {\begin{tabular}{|@{\ }c@{\ }|@{\ }c@{\ }|@{\ }c@{\ }|@{\ }c@{\ }|}
\hline
0 & 0 & 0 & 0\\\hline
1 & 0 & 0 & 0\\\hline
1 & 1 & 0 & 0\\\hline
1 & 1 & 1 & 0\\\hline
\end{tabular}};
\end{tikzpicture}
&
\begin{tikzpicture}[line width=1pt,scale=1.2]
\coordinate (0) at (0,0);
\coordinate (3) at (0,1.3);
\coordinate (1) at (1,0);
\coordinate (2) at (1,1.3);
\node[circle,draw,minimum width=13pt] (a0) at (0) {};
\node[circle,draw,minimum width=13pt] (a1) at (1) {};
\node[circle,draw,minimum width=13pt] (a2) at (2) {};
\node[circle,draw,minimum width=13pt] (a3) at (3) {};
\node at (0) {0};
\node at (1) {1};
\node at (2) {2};
\node at (3) {3};
\draw[->,lightgray] (a1) to[bend right=10] (a0);
\draw[->,lightgray] (a1) to[bend left=10] (a2);
\draw[->,lightgray] (a1) to[bend right=10] (a3);
\draw[->,lightgray] (a0) to[bend left=10] (a2);
\draw[->,lightgray] (a0) to[bend right=10] (a3);
\draw[->,lightgray] (a2) to[bend left=10] (a3);
\draw[->] (a3) to[bend left=10] (a2);
\draw[->] (a3) to[bend right=10] (a1);
\draw[->] (a3) to[bend right=10] (a0);
\draw[->] (a2) to[bend left=10] (a1);
\draw[->] (a2) to[bend left=10] (a0);
\draw[->] (a1) to[bend left=10] (a0);
\end{tikzpicture}
\begin{tikzpicture}
\node at (0,0) {\begin{tabular}{|@{\ }c@{\ }|@{\ }c@{\ }|@{\ }c@{\ }|@{\ }c@{\ }|}
\hline
0 & 0 & 1 & 1\\\hline
2 & 0 & 1 & 1\\\hline
1 & 1 & 0 & 1\\\hline
1 & 1 & 1 & 0\\\hline
\end{tabular}};
\end{tikzpicture}
\\
(a) & (b) & (c)\end{tabular}
\caption{The transitive tournaments $T(2310)$ (a) and $T(0123)$ (b) with their adjacency matrices, and their sum $T^2(2310)$ (c), which is $T(2310)+T(0123)$.}
\label{fig:phi2tourn}
\end{figure}
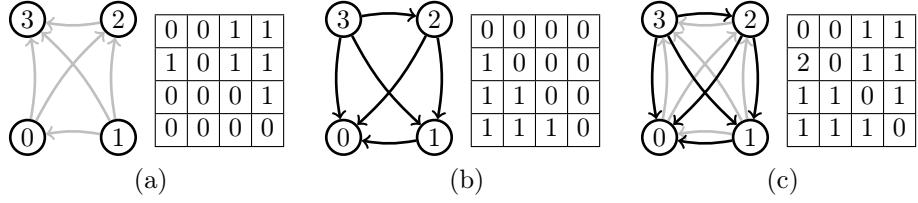

The following is another way to visualize the corresponding $2$-tournament from a permutation arising in Lemma \ref{lem:2tourn}.
Let $\phi\in S_n$.
Define $T$ as the $2$-tournament such that for each $x,y\in Z_n$, 
\begin{equation}
\omega_T(x,y)=\left\{\begin{array}{rl}
0 &\text{if $x < y$ and $\phi(x) < \phi(y)$}, \\
2 &\text{if $x > y$ and $\phi(x) > \phi(y)$, and}\\
1 &\text{otherwise.}\\
\end{array}\right.
\label{eq:otherway}
\end{equation}
That is, two teams $x$ and $y$ split their games if and only if $\phi$ has an inversion between $x$ and $y$; otherwise the team with the greater index wins both games in $T$. 
In the previous example, $2310$ has five inversions; the only non-inversion occurs at indices $0$ and $1$, so $\omega_T(1,0)=2$ and $\omega_T(0,1)=0$.

The methods used in the proof of Lemma \ref{lem:2tourn} are also applicable to Skolem sequences.
Nordh \cite{nordh1} showed a relationship between Skolem sequences and binary xrays, which in our context, is equivalent to the following.
Let $\phi\in S_n$ and suppose $c(\phi)\in \overline{\mathfrak{C}}_n$.
Construct a matrix by replacing each $1$ in $P_\phi$ with the xray index of the $1$.
Project the nonzero entries of this matrix up and to the left. Concatenating these projections, then adding 1 to each term produces a Skolem sequence.
For example, let $\phi=0312$.
Since $C(\phi)=\{0,3,4,5\}$, we appropriately replace the $1$s in $P_\phi$ with the values in $C(\phi)$, then project and add to get the Skolem sequence $64511465$.
See Figure \ref{fig:skolem}.
\begin{figure}
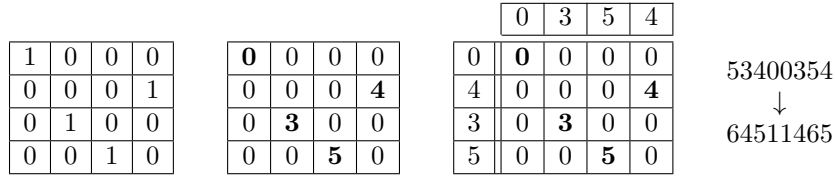

\centering
$\begin{array}[b]{|c|c|c|c|}
\hline
1 & 0 & 0 & 0\\\hline
0 & 0 & 0 & 1\\\hline
0 & 1 & 0 & 0\\\hline
0 & 0 & 1 & 0\\\hline
\end{array}
\qquad
\begin{array}[b]{|c|c|c|c|}
\hline
\mathbf{0} & 0 & 0 & 0\\\hline
0 & 0 & 0 & \mathbf{4}\\\hline
0 & \mathbf{3} & 0 & 0\\\hline
0 & 0 & \mathbf{5} & 0\\\hline
\end{array}
\qquad\begin{array}[b]{|c||c|c|c|c|}
\cline{2-5}
\multicolumn{1}{c|}{} &
0 & 3 & 5 & 4 \\\cline{2-5}
\multicolumn{1}{c}{}\\[-2.2ex]\hline
0 & \mathbf{0} & 0 & 0 & 0\\\hline
4 & 0 & 0 & 0 & \mathbf{4}\\\hline
3 & 0 & \mathbf{3} & 0 & 0\\\hline
5 & 0 & 0 & \mathbf{5} & 0\\\hline
\end{array}
\qquad
\begin{array}[b]{@{}c@{}}
53400354\\
\downarrow \\
64511465\\\\
\end{array}$
\caption{Producing a Skolem sequence from a permutation with a binary xray.}
\label{fig:skolem}
\end{figure}

Skolem sequences can similarly be produced from a $2$-tournament with a TTD.
Let $T\in \mathcal{T}^2_n$ have TTD $T_1+T_2$ for some transitive $T_1,T_2\in \mathcal{T}_n$. 
Let $v_0\cdots v_{n-1}$ and $w_0\cdots w_{n-1}$ be the teams in $Z_n$, ranked from lowest to highest score in $T_1$ and $T_2$, respectively.
Then adding one to each term of the sequence $(s_T(v_{n-1}),\dots,s_T(v_0),s_T(w_0),\dots,s_T(w_{n-1}))$ produces a Skolem sequence.
For example, Let $T\in \mathcal{T}^2_n$ as given in Figure \ref{fig:ktourn}(a), and $T_1$ and $T_2$ be the transitive tournaments
whose adjacency matrices are given in Figures \ref{fig:ktourn}(b) and~\ref{fig:ktourn}(c).
Then $s_T(0)=4$, $s_T(1)=4$, $s_T(2)=3$, $s_T(3)=1$, $v_0v_1v_2v_3=3102$ and $w_0w_1w_2w_3=2301$,
so the resulting Skolem sequence is 
\begin{align*}
s_T(2)s_T(0)s_T(1)s_T(3)s_T(2)s_T(3)s_T(0)s_T(1)\quad+\quad1111111\quad=\quad45524255.
\end{align*}

The previous lemma gives a construction of a $2$-tournament from a permutation, while the next lemma seeks to build a permutation from a $2$-tournament.

\begin{lemma}
\label{lem:ttop}
Let $n$ be a positive integer and $T\in \mathcal{T}_n^2$.
If $T$ has a TTD, then there exists a permutation $\phi\in S_n$ such that $c(\phi)=s(T)$.
\end{lemma}

\begin{proof}
Let $T_1,T_2\in\mathcal{T}_n$ be transitive tournaments such that $T=T_1+T_2$.
Define $\phi = \phi_{T_2}\phi_{T_1}^{-1}$.
So, as multisets,
\begin{align*}
C(\phi) 
&= \{ i + \phi(i): i\in Z_n\} = \{ i + \phi_{T_2}\phi_{T_1}^{-1}(i): i\in Z_n\} \\
&= \{ \phi_{T_1}(j) + \phi_{T_2}\phi_{T_1}^{-1}\phi_{T_1}(j): j\in Z_n\} = \{ \phi_{T_1}(j) + \phi_{T_2}(j): j\in Z_n\}\\
&= \{ s_{T_1}(j) + s_{T_2}(j): j\in Z_n\}= \{ s_{T}(j): j\in Z_n\}
= S(T).\qedhere
\end{align*}
\end{proof}

Note that the converse of Lemma \ref{lem:ttop} is not true.
For example, let $n$ be a positive, odd integer, $n\geq 3$, and $T\in\mathcal{T}_n$ be a regular tournament.
Then $2T\in \mathcal{T}_n^2$ and $s(2T)=(n-1:n)$.
Let $\phi\in S_n$ such that $\phi(i)=n-1-i$ for each $i\in Z_n$.
Then $c(\phi) = (n-1:n)=s(2T)$, but as discussed earlier, since $T$ is not transitive, $2T$ does not have a TTD.

We have thus established that, given a score sequence $s\in\mathfrak{S}_n^2$, there may be tournaments $T\in\mathcal{T}_n^2$ for which $s(T)=s$ and $T$ does not have a TTD.
This leads to the following question:

\begin{question}
Let $n$ be a positive integer and $s\in\mathfrak{S}_n^2$. 
Does there exist $T\in\mathcal{T}_n^2$ such that $T$ has a TTD and $s(T) =s$?
\end{question}

In general, the answer to this question is no.

\begin{example}\label{ex:333777ex}
Let $n=2k$ for an integer $k\geq 3$.
Consider the sequence $s=(k:k)(3k-2:k)$.
\begin{figure}
\centering
$\begin{array}{|c|c|c|c|c|c|}
\hline
0 & 1 & 1 & 1 & 0 & 0\\\hline
1 & 0 & 1 & 0 & 1 & 0\\\hline
1 & 1 & 0 & 0 & 0 & 1\\\hline
1 & 2 & 2 & 0 & 1 & 1\\\hline
2 & 1 & 2 & 1 & 0 & 1\\\hline
2 & 2 & 1 & 1 & 1 & 0\\\hline
\end{array}$
\caption{A $2$-tournament with score sequence $(3,3,3,7,7,7)$.}
\label{fig:333777}
\end{figure}
See Figure \ref{fig:333777} for the adjacency matrix of a $2$-tournament $T_0\in\mathcal{T}_6^2$ such that $s(T_0)=s$ with $k=3$.
It is a straightforward computation to see that, for arbitrary $k\geq 3$, $s$ meets the conditions outlined in \eqref{eq:xraycondition}, and hence $s\in \mathfrak{S}_n^2$.
By way of contradiction, suppose there exists $T\in\mathcal{T}_{n}^2$ with a TTD $T_1+T_2$ for which $s(T)=s$.
We may assume that $T_1=T(\epsilon)$, after an appropriate simultaneous permutation of rows and columns.
So 
\begin{equation*}
S(T_2) = \{s_{T_2}(i):i\in Z_n\} = \{s_T(i)-s_{T(\epsilon)}(i): i\in Z_n\} = \{s_T(i)-i:i\in Z_n\}.
\end{equation*}
Since $T_2$ is transitive, $S(T_2)=Z_n$, so in particular, there are distinct $i,j\in Z_n$ such that $s_T(i)-i = 1$ and $s_T(j)-j = 2k-1$.
Since $s_T(i),s_T(j)\in\{k,3k-2\}$, we have that $i\in\{k-1,3k-3\}$ and $j\in \{-(k-1),k-1\}$.
Since $3k-3,-(k-1)\notin Z_n$ when $k\geq 3$, it follows that $i=j=k-1$, which is a contradiction.
Hence no tournament $T\in \mathcal{T}_n^2$ for which $s(T)=s$ has a TTD.

\end{example}

\subsection{Permutations and 1-tournaments}

We now relate this line of questioning to binary xrays.
Recall that $\alpha:\overline{\mathfrak{C}}_n\to\mathfrak{S}_n$ is an injection, and observe that $c^{-1}(\overline{\mathfrak{C}}_n)$ is the set of all permutations in $S_n$ whose xray is binary.

One way to prove Conjecture \ref{conj:bebeacua} would be to find a subset $\mathcal{T} \subseteq \mathcal{T}_n$ and a function $\delta:\mathcal{T} \to c^{-1}(\overline{\mathfrak{C}}_n)$ such that both $S(\mathcal{T})=\mathfrak{S}_n$ and $s = c\circ \delta$.
While we were not successful in this, we now present a  function $\gamma:c^{-1}(\overline{\mathfrak{C}}_n)\to\mathcal{T}_n$ so that $s\circ \gamma = \alpha \circ c$; that is we factor $c$ through $\mathcal{T}_n$.
The existence of such a function is an immediate corollary of Theorem \ref{thm:landau}; given a permutation $\phi\in c^{-1}(\overline{\mathfrak{C}}_n)$, simply construct a tournament $T$ for which $s(T) = (\alpha\circ c)(\phi)$ using Theorem \ref{thm:landau}.
However, what we present is a function which provides a direct construction of a tournament from a permutation that does not require the use of Theorem \ref{thm:landau}, and imbues some structural properties of $\phi$ in the resulting tournament. 

\begin{theorem}
\label{thm:factoralpha}
Let $n$ be a positive integer.
There exists $\gamma:c^{-1}(\overline{\mathfrak{C}}_n)\to \mathcal{T}_n$ such that for all $\phi\in c^{-1}(\overline{\mathfrak{C}}_n)$, $(s\circ\gamma)(\phi) = (\alpha\circ c)(\phi)$.
\end{theorem}

\begin{proof}
Let $\phi\in c^{-1}(\overline{\mathfrak{C}}_n)$.
Let $T = T^2(\phi)$, and let $s(T) = (s_0,\dots,s_{n-1})$.
Since $s(T)= c(\phi)\in \overline{\mathfrak{C}}_n$, $s_0 <\cdots < s_{n-1}$.
Note that $s(T)$ is presented in increasing order, and does not necessarily reflect the scores in numerical order of teams; let $\rho\in S_n$ be defined so that $s_T(i) = s_{\rho(i)}$, that is, the $i$th team scored the $\rho(i)$th lowest score in $T$.
Recall that, as outlined in \eqref{eq:otherway}, 
if $\rho(i) < \rho(j)$, then $s_{\rho(i)} < s_{\rho(j)}$, and hence $\omega_T(\rho(j),\rho(i)) \geq 1$.
So $T(\rho)$ is a subtournament of $T$.
Define $\gamma(\phi)\in \mathcal{T}_n$ such that $T$ decomposes as $T(\rho)+\gamma(\phi)$.
Therefore
\begin{align*}
S(\gamma(T)) 
&= \{ s_{\gamma(\phi)}(i): i\in Z_n\}
\\&= \{ s_T(i)-s_{T(\rho)}(i): i\in Z_n\}
\\&= \{ s_T(i)-\rho(i): i\in Z_n\}
\\&= \{ s_{\rho(i)}-\rho(i): i\in Z_n\}
\\&= \{ s_{i}-i: i\in Z_n\}.
\end{align*}
So $s(\gamma(\phi)) = \alpha( s(T) ) = \alpha(c(\phi))$.
\end{proof}

An analogous question in the other direction is the following.

\begin{question}
\label{qu:bebeacua}
Let $n$ be a positive integer.
Does there exist a subset $\mathcal{T}\subseteq \mathcal{T}_n$ and a function $\delta:\mathcal{T}\to S_n$ such that 
$s(\mathcal{T})=\mathfrak{S}_n$ and for all $T\in\mathcal{T}$, $(\alpha\circ c\circ\delta)(T) = s(T)$?
\end{question}

If the answer to Question \ref{qu:bebeacua} is yes, then Conjecture \ref{conj:bebeacua} follows.
Indeed, suppose that for every $s\in \mathfrak{S}_n$, there exists $T \in \mathcal{T}_n$ such that $s(T)=s$
and $\alpha^{-1}(s) = (c\circ\delta)(T) \in \overline{\mathfrak{C}}_n$.
Since $\alpha$ is a bijection between $\mathfrak{S}_n$ and $\overline{\mathfrak{S}_n^2}$, it follows that $\overline{\mathfrak{S}_n^2} \subseteq \overline{\mathfrak{C}}_n$.
However $\overline{\mathfrak{C}}_n \subseteq \overline{\mathfrak{S}_n^2}$, so $\overline{\mathfrak{C}}_n = \overline{\mathfrak{S}_n^2}$, and hence $\overline{\mathfrak{C}}_n = \overline{\mathfrak{S}_n^2}$, which proves the conjecture.

In an effort to find $\delta$, the next lemma gives some sufficient conditions for the existence of such a function which aligns with the earlier constructions.

\begin{lemma}
\label{lem:otherway}
Let $n$ be a positive integer.
Suppose there exists $\mathcal{T}\subseteq \mathcal{T}_n$ such that $s(\mathcal{T}) = \mathfrak{S}_n$, and for each $T\in \mathcal{T}$, there exists a transitive $T'\in \mathcal{T}_n$ such that 
$(\alpha\circ s)(T+T') = s(T)$,
$T+T'$ has a TTD, and 
if $x <_{T'} y$ for some $x,y\in Z_n$, then $s_T(x) \leq s_T(y)$.
Then there exists $\delta:\mathcal{T}_n\to S_n$ such that $(\alpha\circ c \circ \delta)(T)=s(T)$.
\end{lemma}

Note that the latter condition relating $T$ to $T'$ is analogous to embedding a partial ordering of $Z_n$ in a total ordering.
For a tournament $W\in\mathcal{T}_n$, let $(Z_n,\prec_W)$ be the partial ordering of $Z_n$ such that for each $x,y\in Z_n$, $x \prec_W y$ if and only if $s_W(x) < s_W(y)$.
Then the latter condition on $T'$ in Lemma \ref{lem:otherway} is that $(Z_n,\prec_{T'})$ is a total ordering in which $(Z_n,\prec_T)$ is embedded.

\begin{proof}
Let $T \in \mathcal{T}$, and let $T'\in \mathcal{T}_n$ be transitive and satisfy the given hypotheses.
Without loss of generality, we relabel the vertices of $T$ and $T'$ so that $T' = T(\epsilon)$; that is $s_{T'}(i)=i$ for each $i\in Z_n$.
We may simply revert to the old labeling at the end of the proof, otherwise.

Let $s(T) = (s_0,\dots,s_{n-1})$ with $s_0 \leq  \cdots \leq s_{n-1}$.
Let $\rho\in S_n$ such that $T' = T(\rho)$.
So if $\rho(x) < \rho(y)$ for some $x,y\in Z_n$, then $s_T(x) \leq s_T(y)$.
Hence $s_{\rho(i)} = s_T(i)$ for each $i\in Z_n$.
Define $T'' = T + T'$.
Then $T''\in \mathcal{T}_n^2$, and 
\begin{align*}
S(T'') 
  &= \{ s_{T}(i) + s_{T'}(i): i\in Z_n\}
\\&= \{ s_{\rho(i)} + \rho(i): i\in Z_n\}
\\&= \{ s_j + j: j\in Z_n\}.
\end{align*}
So $s(T'') = (\alpha^{-1}\circ s)(T)$.
Since $T''$ has a TTD from the hypotheses, by Lemma~\ref{lem:ttop}, there exists $\phi\in S_n$ such that $c(\phi) = s(T'')$.
Define $\delta(T) = \phi$.
Then
\begin{equation*}
  (\alpha\circ c\circ \delta)(T)
= (\alpha\circ c)(\phi)
= (\alpha\circ s)(T'')
= (\alpha\circ \alpha^{-1}\circ s)(T)
= s(T).\qedhere
\end{equation*}
\end{proof}

\begin{example}
\label{ex:tourn2phi}
Let $T\in \mathcal{T}_4$ as given by the matrix in Figure \ref{fig:tourn}(b).
Then $s(T) = (1,1,2,2)$.
Let $T'\in \mathcal{T}_4$ be the transitive tournament given by the matrix in Figure \ref{fig:tourn}(a).
Then $T' = T(2013)$.
Observe that $(Z_n,\prec_T)$ embeds in $(Z_n,\prec_{T'})$.
Then $T+T'\in \mathcal{T}_4^2$, as given in Figure \ref{fig:tourn2phi}(a), and $T+T'$ has a TTD given by $T(1023) + T(3102)$, as given in Figures \ref{fig:tourn2phi}(b) and (c).
Using the method outlined in the proof of Lemma \ref{lem:ttop}, the permutation associated to this TTD is $\phi = 3102 \circ 1023^{-1} = 1302$.
\end{example}
\begin{figure}
\centering
\begin{tabular}{@{}ccc@{}}
\begin{tikzpicture}[line width=1pt,scale=1.2]
\coordinate (0) at (0,0);
\coordinate (3) at (0,1.3);
\coordinate (1) at (1,0);
\coordinate (2) at (1,1.3);
\node[circle,draw,minimum width=13pt] (a0) at (0) {};
\node[circle,draw,minimum width=13pt] (a1) at (1) {};
\node[circle,draw,minimum width=13pt] (a2) at (2) {};
\node[circle,draw,minimum width=13pt] (a3) at (3) {};
\node at (0) {0};
\node at (1) {1};
\node at (2) {2};
\node at (3) {3};
\draw[->,lightgray] (a0) to[bend right=10] (a3);
\draw[->,lightgray] (a0) to[bend left=10] (a1);
\draw[->,lightgray] (a0) to[bend left=10] (a2);
\draw[->,lightgray] (a3) to[bend left=10] (a1);
\draw[->,lightgray] (a3) to[bend right=10] (a2);
\draw[->,lightgray] (a1) to[bend left=10] (a2);
\draw[->] (a3) to[bend left=10] (a2);
\draw[->] (a3) to[bend right=10] (a0);
\draw[->] (a3) to[bend right=10] (a1);
\draw[->] (a2) to[bend left=10] (a0);
\draw[->] (a2) to[bend left=10] (a1);
\draw[->] (a0) to[bend right=10] (a1);
\end{tikzpicture}
\begin{tikzpicture}
\node at (0,0) {\begin{tabular}{|@{\ }c@{\ }|@{\ }c@{\ }|@{\ }c@{\ }|@{\ }c@{\ }|}
\hline
0 & 2 & 1 & 1\\\hline
0 & 0 & 1 & 0\\\hline
1 & 1 & 0 & 0\\\hline
1 & 2 & 2 & 0\\\hline
\end{tabular}};
\end{tikzpicture}
&
\begin{tikzpicture}[line width=1pt,scale=1.2]
\coordinate (0) at (0,0);
\coordinate (3) at (0,1.3);
\coordinate (1) at (1,0);
\coordinate (2) at (1,1.3);
\node[circle,draw,minimum width=13pt] (a0) at (0) {};
\node[circle,draw,minimum width=13pt] (a1) at (1) {};
\node[circle,draw,minimum width=13pt] (a2) at (2) {};
\node[circle,draw,minimum width=13pt] (a3) at (3) {};
\node at (0) {0};
\node at (1) {1};
\node at (2) {2};
\node at (3) {3};
\draw[->] (a3) to[bend left=10] (a2);
\draw[->] (a3) to[bend right=10] (a0);
\draw[->] (a3) to[bend right=10] (a1);
\draw[->] (a2) to[bend left=10] (a0);
\draw[->] (a2) to[bend left=10] (a1);
\draw[->] (a0) to[bend right=10] (a1);
\end{tikzpicture}
\begin{tikzpicture}
\node at (0,0) {\begin{tabular}{|@{\ }c@{\ }|@{\ }c@{\ }|@{\ }c@{\ }|@{\ }c@{\ }|}
\hline
0 & 1 & 0 & 0\\\hline
0 & 0 & 0 & 0\\\hline
1 & 1 & 0 & 0\\\hline
1 & 1 & 1 & 0\\\hline
\end{tabular}};
\end{tikzpicture}
&
\begin{tikzpicture}[line width=1pt,scale=1.2]
\coordinate (0) at (0,0);
\coordinate (3) at (0,1.3);
\coordinate (1) at (1,0);
\coordinate (2) at (1,1.3);
\node[circle,draw,minimum width=13pt] (a0) at (0) {};
\node[circle,draw,minimum width=13pt] (a1) at (1) {};
\node[circle,draw,minimum width=13pt] (a2) at (2) {};
\node[circle,draw,minimum width=13pt] (a3) at (3) {};
\node at (0) {0};
\node at (1) {1};
\node at (2) {2};
\node at (3) {3};
\draw[->,lightgray] (a0) to[bend right=10] (a3);
\draw[->,lightgray] (a0) to[bend left=10] (a1);
\draw[->,lightgray] (a0) to[bend left=10] (a2);
\draw[->,lightgray] (a3) to[bend left=10] (a1);
\draw[->,lightgray] (a3) to[bend right=10] (a2);
\draw[->,lightgray] (a1) to[bend left=10] (a2);
\end{tikzpicture}
\begin{tikzpicture}
\node at (0,0) {\begin{tabular}{|@{\ }c@{\ }|@{\ }c@{\ }|@{\ }c@{\ }|@{\ }c@{\ }|}
\hline
0 & 1 & 1 & 1\\\hline
0 & 0 & 1 & 0\\\hline
0 & 0 & 0 & 0\\\hline
0 & 1 & 1 & 0\\\hline
\end{tabular}};
\end{tikzpicture}
\\
(a) & (b) & (c)\end{tabular}
\caption{The $2$-tournament (a) constructed in Example \ref{ex:tourn2phi}, and two transitive tournaments $T(1023)$ (b) and $T(3102)$ (c) which comprise a TTD of the $2$-tournament.}
\label{fig:tourn2phi}
\end{figure}
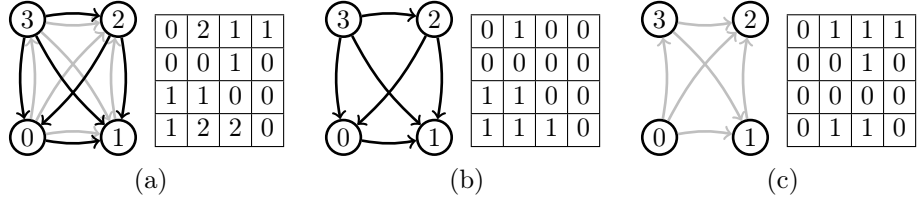

As the next example shows, the choice of $\mathcal{T}$ in Lemma \ref{lem:otherway} must exclude certain tournaments in $\mathcal{T}_n$.

\begin{example}
\label{ex:badfive}
Let $T_1\in\mathcal{T}_5$ be the tournament given in Figure \ref{fig:11233}.
\begin{figure}
\centering
\begin{tabular}{@{}c@{\qquad}c@{}}
\begin{tikzpicture}[line width=1pt,scale=1.2]
\coordinate (0) at ({90-0*72}:1.125);
\coordinate (1) at ({90-1*72}:1.125);
\coordinate (2) at ({90-2*72}:1.125);
\coordinate (3) at ({90-3*72}:1.125);
\coordinate (4) at ({90-4*72}:1.125);
\node[circle,draw,minimum width=13pt] (a0) at (0) {};
\node[circle,draw,minimum width=13pt] (a1) at (1) {};
\node[circle,draw,minimum width=13pt] (a2) at (2) {};
\node[circle,draw,minimum width=13pt] (a3) at (3) {};
\node[circle,draw,minimum width=13pt] (a4) at (4) {};
\node at (0) {0};
\node at (1) {1};
\node at (2) {2};
\node at (3) {3};
\node at (4) {4};
\draw[->] (a4) to (a3);
\draw[->] (a4) to (a2);
\draw[->] (a4) to (a1);
\draw[<-,double] (a4) to (a0);
\draw[->] (a3) to (a2);
\draw[->] (a3) to (a1);
\draw[->] (a3) to (a0);
\draw[->] (a2) to (a1);
\draw[->] (a2) to (a0);
\draw[->] (a1) to (a0);
\end{tikzpicture}
&
\begin{tikzpicture}[line width=1pt,scale=1.2]
\coordinate (0) at ({90-0*72}:1.125);
\coordinate (1) at ({90-1*72}:1.125);
\coordinate (2) at ({90-2*72}:1.125);
\coordinate (3) at ({90-3*72}:1.125);
\coordinate (4) at ({90-4*72}:1.125);
\node[circle,draw,minimum width=13pt] (a0) at (0) {};
\node[circle,draw,minimum width=13pt] (a1) at (1) {};
\node[circle,draw,minimum width=13pt] (a2) at (2) {};
\node[circle,draw,minimum width=13pt] (a3) at (3) {};
\node[circle,draw,minimum width=13pt] (a4) at (4) {};
\node at (0) {0};
\node at (1) {1};
\node at (2) {2};
\node at (3) {3};
\node at (4) {4};
\draw[<-,double] (a4) to (a3);
\draw[->] (a4) to (a2);
\draw[->] (a4) to (a1);
\draw[->] (a4) to (a0);
\draw[<-,double] (a3) to (a2);
\draw[->] (a3) to (a1);
\draw[->] (a3) to (a0);
\draw[<-,double] (a2) to (a1);
\draw[->] (a2) to (a0);
\draw[<-,double] (a1) to (a0);
\end{tikzpicture}
\\
$T_1$ & $T_2$
\end{tabular}
\caption{The two nonisomorphic tournaments with score sequence $(1,1,2,3,3)$.
The edges in bold highlight the edges necessary to be reversed to yield the  transitive tournament $T_{01234}$.}
\label{fig:11233}
\end{figure}
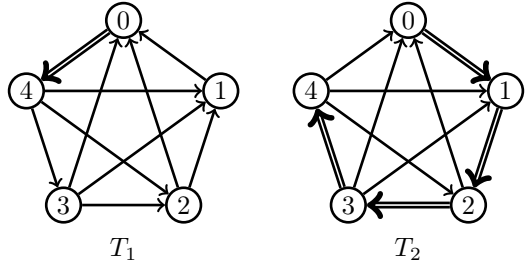
Note that $0 >_{T_1} 4$.
Let $T'\in\mathcal{T}_5$ be a transitive tournament such that, if $s_{T_1}(x)< s_{T_1}(y)$, then $x <_{T'} y$.
Then $T'=T(\phi)$ for some $\phi\in\{01234,10234,01243,10243\}$.
Suppose that $T_1+T'=R_1+R_2$ for transitive $R_1,R_2\in\mathcal{T}_5$, and without loss of generality, assume $0 >_{R_1} 4$.
Observe that regardless of $T'$, $s_{R_1+R_2}(0) \leq 2$ and $s_{R_1+R_2}(4) \geq 6$.
Then $s_{R_1}(0)\leq 2$, and hence $s_{R_1}(4) \leq 1$.
However, this implies $s_{R_2}(4)\geq 5$, which contradicts that $s_{R_2}(4) \leq 4$.
So $T_1+T'$ does not have a TTD, and hence $T_1$ would not belong to the domain of $\delta$.
\end{example}

In the previous example, observe that $s(T_1)=(1,1,2,3,3)$.
In fact, there are two non-isomorphic tournaments $T_1,T_2\in\mathcal{T}_5$ such that $s(T_1)=s(T_2)=(1,1,2,3,3)$.
See Figure \ref{fig:11233}.
Observe that $T_2+T(01234)$ is decomposable as $T(02143)+T(10324)$, and hence could belong to the domain of $\delta$.
So the choice of $\mathcal{T}$ in Lemma \ref{lem:otherway} is critical in this method of proof of Conjecture \ref{conj:bebeacua}.

\subsection[Permutations and k-tournaments]{Permutations and $k$-tournaments}
\label{sec:ktourn}

Recall that for all positive integers $n$, $\mathfrak{S}_n$ is in bijection with $\overline{\mathfrak{S}_n^2}$; this has a natural generalization, which we prove below for completion.
To that end, for positive integers $n$ and $k$, let $\overline{\mathfrak{S}^k_n}$ denote the subset of increasing sequences in~$\mathfrak{S}_n^k$.

\begin{lemma}
\label{lem:ktok-1}
Let $n$ and $k$ be positive integers.
Then $\alpha:\overline{\mathfrak{S}_n^k} \to \mathfrak{S}_n^{k-1}$, where $\alpha(s_0,\dots,s_{n-1}) = (s'_0,\dots,s'_{n-1})$ with $s'_i = s_i - i$ for each $i\in Z_n$, is a bijection.
\end{lemma}

\begin{proof}
Let $s=(s_0,\dots,s_{n-1})\in \overline{\mathfrak{S}_n^k}$.
Then $s$ is nonnegative and increasing, meaning for each $i\in Z_n$, $s_i\geq i$, and hence $\alpha(s)$ is also a nonnegative sequence.
Let $t\in \{1,\dots,n\}$.
Then 
\begin{equation*}
\sum_{i=0}^{t-1} s_i \geq k\cdot\binom t2.
\end{equation*}
So
\begin{equation*}
\sum_{i=0}^{t-1} s'_i 
= \sum_{i=0}^{t-1} (s_i - i) 
= \sum_{i=0}^{t-1} s_i - \binom t2
\geq k\cdot\binom t2 - \binom t2
= (k-1)\cdot\binom t2,
\end{equation*}
with equality when $t=n$.
Hence $\alpha(s)\in \mathfrak{S}_n^{k-1}$, and clearly $\alpha$ is injective.

Now let $s'=(s'_0,\dots,s'_{n-1})\in \mathfrak{S}_n^{k-1}$.
Define $s = (s_0,\dots,s_{n-1})$, where $s_i = s'_i + i$ for each $i\in Z_n$.
Then $s$ is strictly increasing, and for each $t\in \{1,\dots, n\}$,
\begin{equation*}
\sum_{i=0}^{t-1} s_i 
= \sum_{i=0}^{t-1} s'_i + \sum_{i=0}^{t-1} i
= \sum_{i=0}^{t-1} s'_i + \binom t2
\geq (k-1)\cdot\binom t2 + \binom t2
= k\cdot\binom t2,
\end{equation*}
with equality when $t=n$.
Hence $s\in \overline{\mathfrak{S}_n^k}$, and clearly $\alpha(s) = s'$.
So $\alpha$ is a bijection.
\end{proof}

Let $n$ and $k$ be positive integers and $s\in \mathfrak{S}_n^k$.
Let $\{a_1,\dots,a_t\}$ be the largest set of unique terms in $s$.
We define the \emph{shape} of $s$ as the partition $\{p_1,\dots,p_t\}$ of $n$ such that $a_i$ has $p_i$ occurrences in $s$ for each $i\in\{1,\dots,t\}$.
For a partition $\lambda$ of $n$, define $\mathfrak{S}^k(\lambda)$ as the subset of $\mathfrak{S}_n^k$ which contains all sequences with shape $\lambda$, and for set $\Gamma$ of partitions, let $\mathfrak{S}^k(\Gamma)$ as the union of the sets $\mathfrak{S}^k(\lambda)$ for each $\lambda\in \Gamma$.
We similarly define $\mathfrak{C}(\lambda)$ and $\mathfrak{C}(\Gamma)$.
In particular, observe that $\overline{\mathfrak{C}}_n = \mathfrak{C}(1^n)$ and $\overline{\mathfrak{S}_n^k} = \mathfrak{S}^k(1^n)$.

Bebeacua et al.~\cite{bebeacua} made the following conjecture, based solely on the sequences having matching values for small $n$.\footnote{There is a small typo in the original statement of the conjecture by Bebeacua et al.~in that the bijection involves permutations of $S_{2n}$ rather than $S_n$.}
In what follows, when given a multiset $X$ and a positive integer $k$, we denote the multiset union of $k$ copies of $X$ as $\cup_kX$, and let $kX$ denote the multiset obtained by multiplying each element of $X$ by $k$.

\begin{conjecture}
\label{conj:S3Conj}
For every positive integer $n$, $|\mathfrak{C}(2^n)| = |\mathfrak{S}_n^3|$.
\end{conjecture}

An equivalent statement of this conjecture, using the correspondence given in Lemma \ref{lem:ktok-1}, is the following.

\begin{conjecture}\label{conj:D4Conj}
For every positive integer $n$, $|\mathfrak{C}(2^n)| = |\mathfrak{S}^4(1^n)|$.
\end{conjecture}

Conjecture \ref{conj:D4Conj} seems to be the more natural formulation when compared to Conjecture \ref{conj:S3Conj}.
Indeed, suppose $\phi\in S_{2n}$ and $c(\phi)\in\mathfrak{C}(2^n)$.
Then $c(\phi)$ has the form 
$(c_0:2)(c_1:2)\cdots(c_{n-1}:2)$.
Consider the necessarily increasing sequence $(c_0-1,c_1-1,\dots,c_{n-1}-1)$.
For each $t\in\{1,\dots,n\}$,
\begin{equation*}
\sum_{i=0}^{t-1}(c_i-1) = \sum_{i=0}^{t-1}c_i - t 
=\frac12\sum_{i=0}^{t-1}2c_i - t \geq \binom{2t}2 - t = 4\binom t2,
\end{equation*}
with equality when $t=n$.
Hence $(c_0-1,c_1-1,\dots,c_{n-1}-1)\in \mathfrak{S}^4(1^n)$, giving a direct interpretation of these conjectures by relating xrays to score sequences.
This is, in fact, a special case of an application of the following proposition.

\begin{proposition}
\label{prop:RtokR}
Let $n$ be a positive integer and $\lambda$ be a partition of $n$.
Then $|\mathfrak{S}^2(k\lambda)| = |\mathfrak{S}^{2k}(\lambda)|$.
\end{proposition}

\begin{proof}
Let $s\in\mathfrak{S}^{2k}(\lambda)$, and let $s=(s_0,\dots,s_{n-1})$.
Define $s'$ as the sequence with shape $k\lambda$ given by $(s_0+k-1:k)\cdots(s_{n-1}+k-1:k)$.
Let $t\in Z_{kn}$, and let $q$ and $r$ be integers such that $t = qk + r$ and $0\leq r\leq k-1$.
Since $s_0 \leq \cdots \leq s_{q}$ and $s_q(q+1)\geq s_0 + \cdots + s_q \geq k(q+1)q$, we have that $s_q \geq kq$.
Then
\begin{align*}
\sum_{i=0}^{t-1} s'_i 
&= \sum_{i=0}^{qk-1} s'_i + \sum_{i=0}^{r-1} s'_{qk+i}
= k\sum_{i=0}^{q-1}(s_i+(k-1)) + r(s_q+(k-1))
\\&= r \sum_{i=0}^{q} s_{i} + (k-r) \sum_{i=0}^{q-1} s_{i} + qk(k-1) + r(k-1)
\\&\geq rkq(q+1) + (k-r)(kq)(q-1) + qk^{2} - qk + rk -r
\\&= t^{2}-t + r(k-r)\\
&\geq t^2-t,
\end{align*}
with equality when $t=kn$.
Hence $s'\in \mathfrak{S}^2(k\lambda)$, and therefore the map $s\to s'$ is an injection from $\mathfrak{S}^{2k}(\lambda)$ into $\mathfrak{S}^2(k\lambda)$.

Now suppose $s'\in \mathfrak{S}^{2}(k\lambda)$.
Let $S$ be the (multi)set of $n$ elements such that $s'$ is a nondecreasing sequence of the terms in $\cup_kS$; let $\tilde{s} = (\tilde{s}_0,\dots,\tilde{s}_{n-1})$ be the nondecreasing sequence comprised of the elements of $S$.
Then $\tilde{s}$ has shape $\lambda$.
Since $s'\in \mathfrak{S}_{kn}^2$, we have that, for each $t\in\{1,\dots,n\}$, the sum of the first $kt$ terms of $s'$ is at least $(kt)^2-(kt)$ with equality when $t=n$.
Therefore, for each $t\in Z_n$,
\begin{equation*}
\sum_{i=0}^{t-1} \tilde{s}_i = \dfrac1k\sum_{i=0}^{kt-1}s'_i \geq \dfrac1k\left[(kt)^2-kt\right] = kt^2-t.
\end{equation*}
In particular, when $t=1$, we have that $\tilde{s}_0 \geq k-1$, and hence $\tilde{s}_i\geq k-1$ for each $i\in Z_n$.
Define $s$ as the sequence of length $n$ such that $s_i = \tilde{s}_i-k+1$ for each $i\in Z_n$.
Since $s$ is a translation of $\tilde{s}$, it follows that $s$ also has shape $\lambda$.
Furthermore, $s$ is a sequence of nondecreasing integers such that, for each $t\in\{1,\dots,n\}$,
\begin{equation*}
\sum_{i=0}^{t-1} s_i = \left(\sum_{i=0}^{t-1} \tilde{s}_i\right) - t(k-1)
\geq kt^2 - t - t(k-1) 
= 2k\cdot \binom t2,
\end{equation*}
with equality when $t=n$.
Hence $s\in \mathfrak{S}^{2k}(\lambda)$, which implies that the map $s'\to s$ is an injection from $\mathfrak{S}^2(k\lambda)$ into $\mathfrak{S}^{2k}(\lambda)$.
Therefore $|\mathfrak{S}^2(k\lambda)|=|\mathfrak{S}^{2k}(\lambda)|$.
\end{proof}

Observe that for each positive integer $n$, since the sequences in $\mathfrak{C}_n$ satisfy the same conditions as those satisfied by $\mathfrak{S}_n^2$, then necessarily $\mathfrak{C}_n \subseteq \mathfrak{S}_n^2$.
This observation, coupled with the previous proposition, admits the following corollary.

\begin{corollary}
Let $n$ and $k$ be positive integers and $\lambda$ be a partition of $n$.
Then $\mathfrak{C}(\lambda) \subseteq \mathfrak{S}^2(\lambda)$, and hence $\mathfrak{C}(k\lambda) \subseteq \mathfrak{S}^{2}(k\lambda)$, so $|\mathfrak{C}(k\lambda)|\leq |\mathfrak{S}^{2k}(\lambda)|$.
In particular, $|\mathfrak{C}(2^n)| \leq |\mathfrak{S}^4(1^n)|$.
\end{corollary}

The previous corollary makes some inroads on explaining why Conjecture \ref{conj:D4Conj} might be true.
We note that a modest generalization of this conjecture is not true, which we highlight below. 

\begin{corollary}
\label{cor:0k}
Let $n$ and $k$ be positive integers with $n\geq 2$ and $k\geq 3$.
Then $\mathfrak{C}(k^n)$ is a proper subset of $\mathfrak{S}^{2}(k^n)$, and hence $|\mathfrak{C}(k^n)| < |\mathfrak{S}^{2k}(1^n)|$.
\end{corollary}

\begin{proof}
Let $s=(k:k)(3k-2:k)(5k-1:k)(7k-1:k)\cdots(2nk-k-1:k)$.
A direct computation shows that $s\in\mathfrak{S}^2(k^n)$.
Assume that $s\in\mathfrak{C}(k^n)$; then there exists a permutation $\phi\in S_{kn}$ such that $C(\phi)=s$.
Then necessarily, by Lemma~\ref{lem:PermReduc}, $P_\phi$ is a block diagonal matrix $P_{\phi'}\oplus B_k \oplus\cdots\oplus B_k$,
where $\phi'\in S_{2k}$ such that $c(\phi')=(k:k)(3k-2:k)$, and where $B_{k}$ denotes the $k \times k$ backward identity matrix.
However, we showed in Example \ref{ex:333777ex} that no such permutation $\phi'$ exists.
Therefore $s\notin\mathfrak{C}(k^n)$, and hence our assumption is false.
So $\mathfrak{C}(k^n)$ is a proper subset of  $\mathfrak{S}^2(k^n)$.
\end{proof}

It would be ideal to have an understanding of Proposition \ref{prop:RtokR} on the level of tournaments.
That is, we would like to illustrate the bijection given in the proof of Proposition \ref{prop:RtokR} by finding a correspondence between $2k$-tournaments on $Z_n$ and $2$-tournaments on $Z_{kn}$.

Let $n$ be a positive integer, $\lambda$ be a partition of $n$, $T\in \mathcal{T}_{kn}^2$ such that $s(T)\in \mathfrak{S}^2(k\lambda)$.
Then  $S(T) = \cup_kS$ for an integer subset $S=\{s_0,\dots,s_{n-1}\}$ such that $s_0 \leq \cdots \leq s_{n-1}$, and let $s$ be the corresponding nondecreasing sequence associated to $S$.
Then $s$ has shape $\lambda$.
Furthermore, there exists a partition $\{V_i:i\in Z_n\}$ of $Z_{kn}$ such that $|V_i| = k$ and $s_T(a)=s_i$ for each $i\in Z_n$ and $a\in V_i$.
Define $T'$ as the weighted digraph with $V(T') = Z_n$ such that
\begin{equation*}
\omega_{T'}(i,j) = \dfrac1k\sum_{v\in V_i}\sum_{w\in V_j}\omega_T(v,w),
\end{equation*}
with the convention that $\omega_{T'}(i,i)=0$ for all $i\in Z_n$.
Then for all $i,j\in Z_n$,
\begin{align*}
\omega_{T'}(i,j) + \omega_{T'}(j,i)
&= \dfrac1k\sum_{v\in V_i}\sum_{w\in V_j}\omega_T(v,w) + \dfrac1k\sum_{w\in V_j}\sum_{v\in V_i}\omega_T(w,v)
\\&= \dfrac1k\sum_{v\in V_i}\sum_{w\in V_j}\left(\omega_T(v,w)+\omega_T(w,v)\right)
= \dfrac1k(k^2)(2) = 2k.
\end{align*}
We may loosely interpret $T'$ as a $2k$-tournament in the following way.
The vertices of $T'$ correspond to teams of players from $T$ with the same score in $T$, and $\omega_{T'}(i,j)$ is the average number of wins each player in team $i$ has against the players in team $j$.
However, this average may not always yield integral numbers of wins.

Recall that if $s_T(v)=s_T(v')$ for some $v,v'\in Z_{kn}$, then $\omega_T(v,v')=\omega_T(v',v)=1$; that is, the players split their games.
Hence for each $i\in Z_n$ and $v\in V_i$,
\begin{equation*}
\sum_{w\in V_i}\omega_T(v,w) = k-1, \text{ and hence }
\sum_{w\notin V_i}\omega_T(v,w) = s_i-(k-1).
\end{equation*}
We may similarly define the score set of $T'$ as $\{ s_{T'}(i):i\in Z_n\}$, where 
\begin{align*}
s_{T'}(i) 
&= \sum_{\substack{j\in Z_n\\ j\neq i}} \omega_{T'}(i,j)
= \sum_{\substack{j\in Z_n\\ j\neq i}} \dfrac1k\sum_{v\in V_i}\sum_{w\in V_j}\omega_T(v,w)
\\&= \dfrac1k\sum_{v\in V_i}\sum_{w\notin V_i}\omega_T(v,w)
= \dfrac1k\sum_{v\in V_i}(s_i-k+1)
= s_i-k+1.
\end{align*}
Hence $S(T') = \{ s_i-k+1: i\in Z_n\}$, and therefore $s(T')\in \mathfrak{S}^{2k}(\lambda)$.

One might hope that, given $s\in\mathfrak{S}^2(k\lambda)$, there exists a permutation $\phi\in S_{kn}$ such that $c(\phi) = s$ and its corresponding $2$-tournament $T^2(\phi)$ can be used to produce a tournament in $\mathcal{T}_n^{2k}$ with integer scores, but this is not always possible.

\begin{example}
\label{ex:bad02}
Let $s = (2,2,6,6,8,8,12,12)$ and $s'=(1,5,7,11)$; observe that $s\in \mathfrak{C}_{8}(2^4)$ and $s'\in\mathfrak{S}_4^4$.
There exist tournaments $T\in \mathcal{T}_8^2$ and $T'\in\mathcal{T}_4^4$ for which $s(T)=s$ and $s(T')=s'$, and $T$ and $T'$ are related through the previously outlined correspondence; see Figure \ref{fig:t8ex}.
\begin{figure}
\centering
$\begin{array}{|c|c|c|c|c|c|c|c|}\hline
0 & 1 & 0 & 1 & 0 & 0 & 0 & 0\\\hline
1 & 0 & 1 & 0 & 0 & 0 & 0 & 0\\\hline
2 & 1 & 0 & 1 & 0 & 1 & 0 & 1\\\hline
1 & 2 & 1 & 0 & 1 & 0 & 1 & 0\\\hline
2 & 2 & 2 & 1 & 0 & 1 & 0 & 0\\\hline
2 & 2 & 1 & 2 & 1 & 0 & 0 & 0\\\hline
2 & 2 & 2 & 1 & 2 & 2 & 0 & 1\\\hline
2 & 2 & 1 & 2 & 2 & 2 & 1 & 0\\\hline
\end{array}
\qquad
\begin{array}{|c|c|c|c|}\hline
0 & 1 & 0 & 0\\\hline
3 & 0 & 1 & 1\\\hline
4 & 3 & 0 & 0\\\hline
4 & 3 & 4 & 0\\\hline
\end{array}$
\caption{Adjacency matrices for tournaments with score sequences in $\mathfrak{S}^2(2^4)$ and $\mathfrak{S}^4(1^4)$, respectively, which follow the correspondence outlined in Example \ref{ex:bad02}.}
\label{fig:t8ex}
\end{figure}
However, we claim that $T$ does not have a TTD, and in fact,
for any $T\in\mathcal{T}_8^2$ with a TTD and $s(T)=s$, its corresponding digraph $T'$ with outdegree vector $s'$ is not a tournament.

Suppose to the contrary; 
assume there exists $T\in \mathcal{T}_8^2$ such that $T$ has a TTD, $s(T)=s$, and the reduction $T'\in \mathcal{T}_4^4$.
Then by Lemma \ref{lem:ttop}, there exists $\phi\in S_8$ such that $T^2(\phi)=T$.
There are four permutations $\phi$ for which $c(\phi)=s$, namely $\phi\in \{21634705, 25034761, 61034725, 27034165\}$.
However, in each case, the resulting digraph $T'$ on $Z_4$ with outdegrees $\{1,5,7,11\}$ does not have integral edge weights; see Figure \ref{fig:bad02} for the adjacency matrices for each.
\begin{figure}
\centering
\begin{tabular}{cc}
$\begin{array}{|>{\rule[-6pt]{0pt}{17pt}}l|l|l|l|}
\hline
0 & 1 & 0 & 0 \\\hline
3 & 0 & \frac32 & \frac12 \\\hline
4 & \frac52 & 0 & \frac12 \\\hline
4 & \frac72 & \frac72 & 0 \\\hline
\end{array}$
&
$\begin{array}{|>{\rule[-6pt]{0pt}{17pt}}l|l|l|l|}
\hline
0 & \frac12 & \frac12 & 0 \\\hline
\frac72 & 0 & \frac32 & 0 \\\hline
\frac72 & \frac52 & 0 & 1 \\\hline
4 & 4 & 3 & 0\\\hline
\end{array}$ 
\\\\[-2ex]
(a) & (b)
\end{tabular}
\caption{The adjacency matrices for the digraph $T'$ when (a) $\phi=21634705$ or $61034725$, and when (b) $\phi=25034761$ or $27034165$, as outlined in Example \ref{ex:bad02}.}
\label{fig:bad02}
\end{figure}
This suggests that Conjecture \ref{conj:D4Conj} may not have a proof built on finding an appropriate injection from a subset of $S_{2n}$ into $\mathcal{T}_n^4$.
\end{example}

For a final note on TTDs, we say that a $k$-tournament $T\in \mathcal{T}_n^k$ is \defn{stably transitive} if there exist transitive tournaments $T_{1}, \ldots, T_{m}$ so that $T + T_{1} + \cdots + T_{m}$ has a TTD. 
(This terminology is borrowed from the notion of a stably free module.)
We have the following.

\begin{theorem} 
\label{thm:stable}
Let $n$ and $k$ be positive integers. 
Then every $k$-tournament is stably transitive.
\end{theorem}

\begin{proof}

Let $T\in \mathcal{T}_n^k$, and let  
$S$ be the (multi)set of ordered pairs $(x,y)$ for which $0\leq x < y \leq n-1$, and $(x,y)$ has multiplicity $\omega_T(x,y)$ in $S$.
Let $m = |S|$, and label the pairs in $S$ so that $S = \{(x_i,y_i):i\in Z_m\}$.

Then, let $\sigma$ be the transposition $(01) \in S_n$, and for $i \in Z_{m}$, let $\phi_{i}$ be a permutation in $S_n$ with $\phi_i(x_i) = 1$ and $\phi_i(y_i) = 0$. 
We see that $\omega_{T(\phi_i)}(x_i,y_i) = 1$ and $\omega_{T(\sigma \phi_i)}(x_i,y_i) = 0$, and the weights of all other edges in $T(\phi_i)$ and $T(\sigma \phi_i)$ are identical. 
Let $T' = T(\phi_0) + \cdots + T(\phi_{m-1})$ and $T'' = T(\sigma\phi_0) + \cdots + T(\sigma\phi_{m-1})$.
Therefore, if $0\leq x < y \leq n-1$, we have that $\omega_{T}(x,y) = \omega_{T'}(x,y) - \omega_{T''}(x,y)$, and furthermore, $\omega_{kT(\epsilon)}(x,y) = 0$.
It follows that 
\begin{align*} 
\omega_{T+T''}(x,y)  
& =  \omega_{T}(x,y) + \omega_{T''}(x,y)
 = \omega_{T'}(x,y)
\\& = \omega_{kT(\epsilon)}(x,y) + \omega_{T'}(x,y)
 = \omega_{kT(\epsilon) + T'}(x,y) 
 \end{align*}
Thus, $T + T'' = kT(\epsilon) + T'$.
Since $kT(\epsilon) + T'$ has a TTD, $T$ is stably transitive.
\end{proof}




Furthermore, given any two tournaments $T,U\in \mathcal{T}_n^k$, there exist tournaments $\widetilde{T}$ and $\widetilde{U}$ with TTDs such that $T + \widetilde{T} = U + \widetilde{U}$.
In fact, if $T',U',T'',U''$ are tournaments with TTDs such that $T + T'' = kT(\epsilon) + T'$ and $U + U'' = kT(\epsilon) + U'$, then we let $\widetilde{T} = U' + T''$ and $\widetilde{U} = T' + U''$.


Given that all tournaments are stably transitive, we have the following open question.

\begin{question}
Let $n$ and $k$ be positive integers and let $T\in \mathcal{T}_n^k$.
What is the smallest integer $m(T)$ such that there exists $T'\in \mathcal{T}_n^m$ such that $T'$ and $T+T'$ have TTDs?
\end{question}

If $m(T)=0$ for a tournament $T$, then $T$ has a TTD, however as illustrated in Examples \ref{ex:333777ex}, there exist some tournaments $T$ for which $m(T)$ is at least $1$.
%
We finish the section with an additional open question.
\begin{question}
Let $n$ and $k$ be positive integers.
Define $m(n,k) = \max\{ m(T): T\in \mathcal{T}_n^k\}$.
What is $m(n,k)$, and how does $m(n,k)$ grow with $n$ or $k$?
\end{question}

It can be shown exhaustively that $m(n,1)=1$ if $1 \leq n \leq 5$, and it follows from Example \ref{ex:333777ex} that $m(n,1)>1$ if $n\geq 6$.
For positive integers $n$ and $k$, a trivial upper bound arising from Theorem \ref{thm:stable} is that $m(n,k) \leq \frac k2\binom n2$.

\section{Restricted score sequences and xrays}
\label{sec:restricted}

As highlighted in Example \ref{ex:333777ex}, we have that $\mathfrak{C}_n \subseteq \mathfrak{S}_n^2$, and $\mathfrak{C}_n \neq \mathfrak{S}_n^2$ if and only if $n\geq 6$.
However, we can find some cases of equality when restricted to score sequences of particular shapes.
In this section, we provide evidence for a conjecture which classifies all integers $n$ and partitions $\lambda$ of $n$ for which $\mathfrak{C}(\lambda) = \mathfrak{S}^2(\lambda)$.
To that end, we say that a partition  $\lambda$ is \emph{incomplete} if $\mathfrak{C}(\lambda) \neq \mathfrak{S}^2(\lambda)$, and hence we seek to determine all complete partitions.
We begin with an inheritance property, as well as a property of characteristic sets of permutations.

\begin{lemma}
\label{lem:rprime}
Let $\lambda$ be a partition of a positive integer $n$ and let $\lambda'\subseteq \lambda$ be a partition of an integer $n'\leq n$.
If $\lambda$ is complete, then $\lambda'$ is complete.
\end{lemma}

\begin{proof}
The result holds if $\lambda = \lambda'$, and thus it is sufficient to show the result holds when $\lambda \backslash \lambda'$ has one element; let $\lambda' \cup \{p\} = \lambda$.
Suppose that $\lambda$ is complete.
Let $s\in\mathfrak{S}^2(\lambda')$.
A direct computation shows that $s(2n'+p-1:p)\in \mathfrak{S}^2(\lambda)$, so there exists $\phi\in S_n$ such that $c(\phi)=s$.
It follows from Lemma \ref{lem:PermReduc} that $P_\phi = P_{\phi'}\oplus B_p$ for some permutation $\phi'\in S_{n'}$, and necessarily $c(\phi') = s$.
So $s\in\mathfrak{C}(\lambda')$, and hence $\lambda'$ is complete.
\end{proof}

\begin{lemma}
\label{lem:handyresult'}
Let $n$ be a positive integer, $\phi\in S_n$, and $d\in Z_n$.
Let $M = \max C(\phi)$, and $\sigma = |\{x\in C(\phi): x \leq d\}|$.
Then $n+\sigma \leq \max\{ M+1, 2d+2 \}$.
\end{lemma}

\begin{proof}
Let $X = Z_n-Z_{d+1}$ and $Y = \{ x\in Z_n: \phi(x)\in X\}$.
Then $|X| = |Y| = n - 1 - d$, and $x+\phi(x) > d$ for each $x\in X\cup Y$.
So $|X\cup Y| \leq n-\sigma$, and therefore $n + \sigma \leq 2n - |X\cup Y|$.
Observe that, if $M \leq 2d+1$, then $X\cap Y$ is empty; otherwise $|X\cap Y| \leq M - (2d+1)$.
Hence, $|X\cap Y| \leq \max\{ 0, M - (2d+1) \}$.
So
\begin{align*}
n + \sigma 
&\leq 2n - |X\cup Y|
= 2n - |X| - |Y| + |X\cap Y| \\
&\leq 2n - 2n+2+2d + \max\{ 0, M-(2d+1)\}\\
&= \max\{ 2d+2, M+1 \}.\qedhere
\end{align*}
\end{proof}

The contrapositive of Lemma \ref{lem:handyresult'} is useful to determine the incompleteness of partitions.
To that end, suppose $\lambda$ is a partition of a positive integer $n$ and $s\in\mathfrak{S}^2(\lambda)$, and let $M$ be the maximum term in $s$.
If there exists $d\in Z_n$ such that
$n+|\{i:s_i\leq d\}| > \max\{ 2d+2, M+1\}$, we say $s$ is \emph{deficient}.
We may restate Lemma \ref{lem:handyresult'} in terms of deficient sequences.

\begin{lemma}
\label{lem:handyresult}
Let $\lambda$ be a partition of a positive integer $n$.
If $\mathfrak{S}^2(\lambda)$ contains a deficient sequence, then $\lambda$ is incomplete.
\end{lemma}

See Table \ref{tab:counterex} for deficient sequences for various partitions.
\begin{table}
\centering
$\begin{array}{|>{\rule[-5pt]{0pt}{16pt}}c|c|c|c|c|c|c|c|}
\hline
n
& \lambda
& s\in\mathfrak{S}^2(\lambda)
& d
& \sigma
& n+\sigma
& 2d + 2
& M + 1 \\\hline
7 & \{1,2,4\} & (4\dv4)(8\dv1)(9\dv2) & 4 & 4 & 11 & 10 & 10 \\\hline
9 & \{1,2,6\} & (6\dv6)(10\dv1)(13\dv2) & 6 & 6 & 15 & 14 & 14 \\\hline
9 & \{2,3,4\} & (5\dv4)(10\dv3)(11\dv2) & 5 & 4 & 13 & 12 & 12 \\\hline
13 & \{2^4,5\} & (6\dv2)(7\dv2)(8\dv2)(16\dv5)(17\dv2) & 8 & 6 & 19 & 18 & 18 \\\hline
\end{array}$
\caption{Examples of sequences in $\mathfrak{S}^2(\lambda)$ which are not in $\mathfrak{C}(\lambda)$, hence demonstrating that $\lambda$ is not complete.}
\label{tab:counterex}
\end{table}
A majority of the results which follow demonstrate the incompleteness of families of partitions by showing the existence of deficient sequences.

\subsection{Partitions with at most two part sizes}

We first observe that the sequence $(k:k)(3k-2:k)$, as given in in Example \ref{ex:333777ex}, is deficient with $d=k$, which is a fact that, when coupled with Lemma \ref{lem:rprime}, classifies several incomplete partitions with only one part size.
In fact, the proof of Corollary \ref{cor:0k} follows this reasoning as well.

\begin{proposition}
\label{prop:k}
Let $k$ and $n$ be positive integers.
Then $\{k\}$ is complete, and if $k\geq 3$, then $k^n$ is incomplete if and only if $n\geq 2$.
\end{proposition}

The next four propositions identify a series of incomplete partitions which have exactly 2 distinct part sizes.




\begin{proposition}
\label{prop:knk}
Let $a$ and $b$ be distinct, positive integers such that $a<b$.
Then $\{a,b\}$ is complete if and only if $\gcd(a,b)\leq 2$.
\end{proposition}

\begin{proof}
Let $g = \gcd(a,b)$, and let $s \in \mathfrak{S}^2(\{a,b\})$.
Then there exist integers $x$ and $y$ such that $xa + yb = (a+b)(a+b-1)$, where $a-1 \leq x \leq a+2b-1$ and $b-1 \leq y \leq 2a+b-1$.
In fact, $x$ and $y$ must be of the form
\begin{equation*}
x = a+b-1+tb/g\text{ and }y = a+b-1 - ta/g,
\end{equation*}
where $t\in\{\pm1,\dots,\pm g\}$.
In particular, if $t=g$ or $t=-g$, then $c(\phi)=s$ if $P_\phi = B_b\oplus B_a$ or $P_\phi = B_a\oplus B_b$, respectively.
If $2$ divides $g$ and $t=g/2$ or $t=-g/2$, then $c(\phi)=s$ if $P_\phi = B_2\otimes(B_{b/2}\oplus B_{a/2})$ or $P_\phi = B_2\otimes(B_{a/2}\oplus B_{b/2})$, respectively.
Hence, $s\in\mathfrak{C}(\{a,b\})$, and therefore $\{a,b\}$ is complete if $g\leq 2$.

Now suppose $g\geq 3$.
Consider the sequence $s$ given by
\begin{equation*}
s = \left(a-1+b/g: a\right)\left(2a+b-1-a/g:b\right),
\end{equation*}
A direct computation shows that $s\in\mathfrak{S}^2(\{a,b\})$.
Let $d=a-1+b/g$; 
then $d\leq a+b-1$ and 
$2a+b > \max\{ 2a-1+2b/g, 2a+b-a/g \}$.
So $s$ is deficient, and therefore $\{a,b\}$ is incomplete by Lemma \ref{lem:handyresult}.
\end{proof}

\begin{proposition}
\label{prop:1k}
Let $k\geq 6$.
Then $\{1^2,k\}$ is complete and $\{1^3,k\}$ is incomplete.
\end{proposition}

\begin{proof}
To show $\{1^2,k\}$ is complete, it is sufficient to show that for each sequence $s\in\mathfrak{S}^2(\{1^2,k\})$ whose modal term is at least $k+1$, there exists $\phi\in S_{k+2}$ such that $c(\phi)=s$.
The sequences in $\mathfrak{S}^2(\{1^2,k\})$ for which the modal term is at least $k+1$ are of the form
\begin{align*}
s^0 &= (0:1)(2:1)(k+3:k), \\
s^1_i &= (i:1)(k+2-i:1)(k+2:k),\text{ or}\\
s^2_j &= (k+1-j:1)(k+1:k)(k+1+j:1),
\end{align*}
for some integer $i$ such that $1\leq i \leq (k+1)/2$, or $j\in \{1,\dots,k+1\}$.
Let $\alpha,\beta,\gamma\in S_{k+2}$ such that $P_\alpha = I_2\oplus B_{k}$, $P_\beta = I_1\oplus B_{k+1}$, and $P_\gamma = B_{k+2}$.
Let $\sigma_i$ be the transposition $(0i)$ for each $i\in Z_{k+2}$.
Then, for each integer $i$ such that $1\leq i \leq (k+1)/2$ and $j\in\{1,\dots,k+1\}$,
$s^0 = c(\alpha)$,
$s^1_i = c(\sigma_i\beta)$, and 
$s^2_j = c(\sigma_j\gamma)$.
Therefore $\{1^2,k\}$ is complete.

Define the set $S$ as 
\begin{equation*}
S = \left\{\begin{array}{ll}
\{(5k+3)/3, (5k+6)/3,  (5k+9)/3\} & \text{if $k\equiv 0\!\!\!\pmod{3}$},\\
\{(5k+1)/3, (5k+7)/3, (5k+10)/3\} & \text{if $k\equiv 1\!\!\!\pmod{3}$, and}\\
\{(5k+2)/3, (5k+5)/3, (5k+11)/3\} & \text{if $k\equiv 2\!\!\!\pmod{3}$}.
\end{array}\right.
\end{equation*}
Let $s$ be the nondecreasing sequence corresponding to the set $\cup_k\{k\}\cup S$; a direct computation shows that $s\in \mathfrak{S}^2(\{1^3,k\})$.
Since $k\geq 6$, it follows that 
$2k+3 > \max\{2k+2,(5k+14)/3\}$, and hence $s$ is deficient, with $d=k$.
Therefore $\{1^3,k\}$ is incomplete.
\end{proof}

\begin{proposition}
\label{prop:2k}
For all $k\geq 3$, $\{2^2,2k\}$ is incomplete.
\end{proposition}

\begin{proof}
Suppose $k\geq 3$.
Define the set $S$ as
\begin{align*}
S &= 
\left\{ \begin{array}{ll}
\{(7k+4)/2, (7k+8)/2\} & \text{ if $k$ is even,}\\
\{(7k+5)/2, (7k+7)/2\} & \text{ if $k$ is odd.}
\end{array}\right.
\end{align*}
Let $s$ be the nondecreasing sequence corresponding to the set $\cup_{2k}\{2k\}\cup (\cup_2(S))$; a direct computation shows that $s\in \mathfrak{S}^2(\{2^2,2k\})$.
Since $k\geq 3$, we have that $4k+4 > \max\{ 4k+2,(7k+10)/2\}$; 
so $s$ is deficient, with $d=2k$.
Therefore $\{2^2,2k\}$ is incomplete by Lemma \ref{lem:handyresult}.
\end{proof}

\begin{proposition}
\label{prop:22k+1}
The partitions $\{2^3,5\}$ and $\{2^4,5\}$ are complete and incomplete, respectively, and for all $k\geq 3$, $\{2^2,2k+1\}$ and $\{2^3,2k+1\}$ are complete and incomplete, respectively.
\end{proposition}

\begin{proof}
An exhaustive search shows that $|\mathfrak{S}^2(\{2^3,5\})|=|\mathfrak{C}(\{2^3,5\})|=146$, so $\{2^3,5\}$ is complete.
A deficient sequence in $\mathfrak{S}^2(\{2^4,5\})$ is given in Table~\ref{tab:counterex}, and hence $\{2^4,5\}$ is incomplete by Lemma \ref{lem:handyresult}.

Let $k\geq 3$.
To show completeness of $\{2^2,2k+1\}$, it is sufficient to show that for each sequence $s$ of $\mathfrak{S}^2(\{2^2,2k+1\})$ whose modal term is at least $2k+4$, there exists $\phi\in S_{2k+5}$ such that $c(\phi)=s$.
The sequences in $\mathfrak{S}^2(\{2,2k+1\})$ for which the modal term is at least $2k+4$ are of the form
\begin{align*}
s^0_0 &= (1:2)(5:2)(2k+8:2k+1), \\
s^0_1 &= (2:2)(4:2)(2k+8:2k+1), \\
s^1_i &= (i:2)(2k+7-i:2)(2k+6:2k+1),\text{ or}\\
s^2_j &= (2k+4-j:2)(2k+4:2k+1)(2k+4-j:2),
\end{align*}
where $i\in\{2,\dots,k+3\}$, and $j\in\{1,\dots,2k+3\}$.
Let $\alpha,\beta,\gamma,\delta\in S_{2k+5}$ such that $P_\alpha = B_{2}\oplus B_2\oplus B_{2k+1}$, $P_\beta = (B_2\otimes I_2)\oplus B_{2k+1}$, $P_\gamma = B_2\oplus B_{2k+3}$, and $P_\delta = B_{2k+5}$.
Let $\sigma_i$ be the transposition $(0i)$ for each $i\in Z_{2k+5}$.
Then, for each $i\in\{2,\dots,k+2\}$ and $j\in\{1,\dots,2k+3\}$,
$s^0_0 = c(\alpha),\ s^0_1 = c(\beta)$, 
$s^1_i = c(\sigma_i\gamma\sigma_i)$, 
$s^1_{k+3} = c(\sigma_{2k+4}\gamma\sigma_{2k+4})$, and 
$s^2_j = c(\sigma_j\delta\sigma_j)$.
Therefore, $\{2^2,2k+1\}$ is complete.

Define the set $S$ as
\begin{align*}
S &= 
\left\{ \begin{array}{ll}
\{(10k + \gamma)/3: \gamma\in\{15,21,24\}\} & \text{ if } k\equiv 0\!\!\!\pmod{3},\\
\{(10k + \gamma)/3: \gamma\in\{17,20,23\}\} & \text{ if $k\equiv 1\!\!\!\pmod{3}$, and}\\
\{(10k + \gamma)/3: \gamma\in\{16,19,25\}\} & \text{ if } k\equiv 2\!\!\!\pmod{3}.
\end{array}\right.
\end{align*}
Let $s$ be the nondecreasing sequence corresponding to the set $\cup_{2k+1}\{2k+2\}\cup (\cup_2S)$; a direct computation shows that $s\in \mathfrak{S}^2(\{2^3,2k+1\})$.
Since $k\geq 3$, it follows that $4k+8 > \max\{ 4k+6,(10k+28)/3\}$,
and hence $s$ is deficient with $d=2k+2$.
Therefore $\{2^3,2k+1\}$ is incomplete by Lemma \ref{lem:handyresult}.
\end{proof}

We summarize the results of this section with the following.
Let $\Gamma_1$ and $\Gamma_2$ be the sets of partitions given by 
\begin{align*}
\Gamma_1 &= \{ \{a\}, \{1^a\}, \{2^a\}: a\geq 1 \}\text{ and }\\
\Gamma_2 
&= \ \ \,\,\ \{ \{a,b\}: a,b\geq 1,\ \gcd(a,b)\leq 2\} 
\cup \{ \{1^a,2^b\}: a\geq 1,\ b\geq 1\} 
\\&\phantom{=\ \ }\cup \{ \{1^a,b\}: a\geq 2,\ 3\leq b\leq 5\} 
 \cup \{ \{1^2,b\}: b\geq 3\} 
\\&\phantom{=\ \ } \cup \{ \{2^a,b\}: a\geq 2,\ b\in\{3,4\}\} 
 \cup \{ \{2^2,b\}: b\geq 5\text{ and odd}\}
\cup \{ \{2^3,5\}\}.
\end{align*}

\begin{theorem}
\label{thm:12}
Let $\lambda$ be a complete partition of a positive integer $n$.
If $\lambda$ contains exactly one distinct part size, then $\lambda\in \Gamma_1$.
If $\lambda$ contains exactly two distinct part sizes, then $\lambda \in \Gamma_2$.
\end{theorem}

\begin{proof}
If $\lambda$ has only one distinct part size, the result follows directly from Proposition \ref{prop:k}.
Now suppose $\lambda$ has exactly two distinct part sizes $p$ and $q$ with $p < q$.

First, suppose $p\geq 3$. 
Then $\{p^2\}$ and $\{q^2\}$ are not contained in $\lambda$, by Proposition \ref{prop:k} and Lemma \ref{lem:rprime}.
So $\lambda = \{p,q\}$.
So by Proposition \ref{prop:knk}, $\lambda\in \Gamma_2$.

Now, suppose $p=2$.
By Propositions \ref{prop:knk}, \ref{prop:2k} and \ref{prop:22k+1}, it follows that $\lambda\in\Gamma_2$.

Last, suppose $p=1$, and hence $\lambda = \{1^a,q^b\}$.
If $q\geq 3$, then necessarily $b=1$, and by Proposition \ref{prop:1k}, $\lambda\in \Gamma_2$.
Finally, if $q=2$, then $\lambda\in \Gamma_2$ as well.
\end{proof}

\subsection{Partitions with at least three part sizes}

Observe that $|\mathfrak{S}^2(\{1,2,3\})|=|\mathfrak{C}(\{1,2,3\})|=18$, so $\{1,2,3\}$ is complete, and in what follows, we show this is only complete partition with at least three distinct parts.

We begin with first classifying the completeness of partitions with three parts for which the sum of smallest two parts divide the largest part.
Then, we give some number-theoretic results, which we then leverage to classify the complete partitions with three or more parts.

\begin{proposition}
\label{prop:aa12a1}
The partition $\{a,a+1,2a+1\}$ is complete if and only if $a=1$.
The partition $\{a,a+2,2a+2\}$ is incomplete for all positive integers $a$.
\end{proposition}

\begin{proof}
If $a=1$, then $\{a,a+1,2a+1\}=\{1,2,3\}$, which is complete.
Now, suppose $a\geq 2$.
Define the sequence $s$ as 
$(a:a)(3a+1:a+1)(6a+1:2a+1)$.
A direct computation shows that $s\in\mathfrak{S}^2(\{a,a+1,2a+1\})$.
Suppose there exists $\phi\in S_{4a+2}$ such that $c(\phi)=s$.
Then $\phi(a+1)+(a+1)\in\{a,3a+1,6a+1\}$ and $\phi(a+1)\in Z_{4a+2}$, so $\phi(a+1)=2a$ since $a\geq 2$.
Similarly, $\phi(4a+1)+4a+1\in \{a,3a+1,6a+1\}$ and $\phi(4a+1)\in Z_{4a+2}$, so necessarily $\phi(4a+1) = 2a$ as well, which is a contradiction, and therefore $\{a,a+1,2a+1\}$ is incomplete when $a\geq 2$.

In a similar manner, for all $a\geq 1$, the sequence $s'=(a:a)(3a+2:a+2)(6a+4:2a+2)$ belongs to $\mathfrak{S}^2(\{a,a+2,2a+2\})$ and, if there exists $\phi'\in S_{4a+4}$ such that $c(\phi')=s'$, then $\phi(a+1)=\phi(4a+3)=2a+1$, which is a contradiction.
So $\{a,a+2,2a+2\}$ is incomplete for all $a\geq 1$.
\end{proof}

\begin{proposition}
\label{prop:abaplusb}
Let $\lambda$ be a partition of the form $\{a,b,k(a+b)\}$ for positive integers $a,b,k$ with $a < b$. Then $\lambda$ is complete if and only if $\lambda = \{1,2,3\}$.
\end{proposition}

\begin{proof}
Suppose $\lambda \neq \{1,2,3\}$.
By Proposition \ref{prop:aa12a1}, the result holds whenever $k=1$ and $b\in\{a+1,a+2\}$.
Hence, we may assume that either $k\geq 2$ or $b\geq a+3$.
These assumptions imply that either $b=k=2$ or $(k-2)a+k(b-2)>0$.
If $b=k=2$, then $a=1$, and a deficient sequence of $\{1,2,6\}$ is given in Table \ref{tab:counterex}, so we may assume $(k-2)a+k(b-2)>0$, which is equivalent to $(k+2)(a+b) > 4a+2b+2k$.
Furthermore $2a+b+k-1 \leq 2a + kb + ka - 1 + 2b = (k+2)(a+b)-1$.

Let $n=(k+1)(a+b)$, and consider the sequence given by \[ s = (a+k-1:a)(2a+b+k-1:b)((k+2)(a+b)-2:k(a+b)).\] 
A direct computation shows that $s \in \mathfrak{S}^{2}(\{a,b,k(a+b)\})$. 
It follows that $(k+2)(a+b) > \max\{4a+2b+2k,(k+2)(a+b)-1\}$; so $s$ is deficient, with $d = 2a+b+k-1$.
So $\{a,b,k(a+b)\}$ is incomplete by Lemma \ref{lem:handyresult}.
%
\end{proof}

\begin{lemma}
\label{lem:rmwork}
Let $a$, $b$, and $c$ be distinct positive integers such that $a < b < c$, 
\begin{compactenum}[\rm(a)]
\item If $c\geq 5$, then $(c+ab)/(a+b) < c/2$.
\item If $a$ and $b$ are even, $c$ is odd, and $c\geq 7$, then $(2c+ab)/(a+b) < (c+1)/2$. 
\end{compactenum}
\end{lemma}

\begin{proof}[Proof of \rm(a).]
Suppose $c\geq 5$.
Then $(c+2)/3 < c/2$ and $(c+6)/5 < c/2$.
So if $(a,b)\in\{(1,2),(2,3)\}$, then $(c+ab)/(a+b)<c/2$.

Now, suppose $(a,b)\neq (1,2)$ and $(a,b)\neq (2,3)$.
Let $m$ and $r$ be defined such that $a = m-r$ and $b=m+r$.
Then either $r\geq 1$ and $m\geq 2$, or $r=1/2$ and $m\geq 7/2$.
Observe that, under these conditions, we have that $r^2+rm-r-1 > 0$, and additionally, $c \geq m+r+1$.
So
\begin{align*}
\dfrac{c+ab}{a+b}
 &=\dfrac{c+m^2-r^2}{2m}
=\dfrac{m^2-r^2-(m-1)c}{2m} + \dfrac c2\\
&\leq \dfrac{m^2-r^2-(m-1)(m+r+1)}{2m} + \dfrac c2\\
&=\dfrac{r+1-rm-r^2}{2m} + \dfrac c2
< \dfrac c2.\qedhere
\end{align*}
\end{proof}
\begin{proof}[Proof of \rm(b).]
Suppose $c\geq 7$.
Then $(2c+8)/6 < (c+1)/2$, so $(2c+ab)/(a+b)<(c+1)/2$ if $(a,b)=(2,4)$.
Now, suppose $(a,b)\neq (2,4)$.
Again, let $m$ and $r$ be defined such that $a = m-r$ and $b=m+r$.
Then $m$ and $r$ are integers such that 
$m\geq 4$ and $r\geq 1$.
Observe that, under these conditions, we have that $r^2+rm-2r-2 > 0$, and similarly, $c\geq m +r+1$.
So
\begin{align*}
\dfrac{2c+ab}{a+b}
&=\dfrac{2c+m^2-r^2}{2m} 
=\dfrac{m^2-r^2-m-c(m-2)}{2m} + \dfrac{c+1}2 \\
&\leq\dfrac{m^2-r^2-m-(m+r+1)(m-2)}{2m} + \dfrac{c+1}2 \\
&=\dfrac{-r^2-rm+2r+2}{2m} + \dfrac{c+1}2
<\dfrac{c+1}2.\qedhere
\end{align*}
\end{proof}

\begin{lemma}
\label{lem:ntresult}
Let $a$, $b$, and $c$ be positive integers such that $a<b<c$ and $a+b$ does not divide $c$.
There exist distinct integers $x$ and $y$ such that $-a \leq y \leq (c-1)/2$, $-b\leq x \leq (c-1)/2$, and 
\begin{compactenum}[\rm(a)]
\item if $\gcd(a,b)=1$ and $c\geq 5$, then $ax+by = c$; 
\item if $\gcd(a,b)=2$, $c$ is odd, and $c\geq 7$, then $ax+by = 2c$.
\end{compactenum}
\end{lemma}

\begin{proof}[Proof of \rm(a).]
Suppose $\gcd(a,b)=1$ and $c\geq 5$.
Then there exist integers $x$ and $y$ such that $ax + by = c$.
Furthermore, since $a(x-b)+b(y+a)=c$ as well, we may conclude after some iterations, that $-b \leq x - y \leq a$; note that $x\neq y$ since $a+b$ does not divide $c$.

Since $x-y \leq a$ and $ax+by=c$, we have that $(c-by)/a \leq a+y$, or equivalently $y \geq (c-a^2)/(a+b)$,  and hence
\begin{align*}
y \geq \dfrac{c-a^2}{a+b} 
\geq \dfrac{-a^2}{a+b} 
\geq -a\left(\dfrac{a}{a+b}\right)
\geq -a.
\end{align*}
Similarly, we have that $(c-ax)/b \geq x - a$, or equivalently $x \leq (c+ab)/(a+b)$.
In addition, since $y-x \leq b$, we also have that $x\geq -b$ and $y \leq (c+ab)/(a+b)$.
From Lemma \ref{lem:rmwork}, we have that $x<c/2$ and $y<c/2$, and since $x$ and $y$ are integers, $x\leq (c-1)/2$ and $y\leq (c-1)/2$.
\end{proof}
\begin{proof}[Proof of \rm(b).]
Suppose $\gcd(a,b)=2$, $c$ is odd, and $c\geq 7$.
Then there exist integers $x$ and $y$ such that $ax + by = 2c$.
Furthermore, since $a(x-b)+b(y+a)=2c$ as well, we may conclude after some iterations, that $-b \leq x - y \leq a$; again, note that $x\neq y$ since $a+b$ does not divide $c$.

Through similar reasoning in the previous proof, we may deduce that 
$-b \leq x \leq (2c+ab)/(a+b)$ and $-a \leq y \leq (2c+ab)/(a+b)$.
By Lemma \ref{lem:rmwork}, we have that $x < (c+1)/2$ and $y< (c+1)/2$.
Since $c$ is odd and $x$ and $y$ are both integers, it follows that $x\leq (c-1)/2$ and $y\leq (c-1)/2$.
\end{proof}

\begin{proposition}
\label{prop:abc}
Let $a,b,c$ be distinct, positive integers such that $a < c$ and $b < c$.
Then $\{a,b,c\}$ is complete if and only if $\{a,b,c\}=\{1,2,3\}$.
\end{proposition}

\begin{proof}
Suppose $\{a,b,c\}$ is complete.
Then $\gcd(a,b)\leq 2$; otherwise $\{a,b\}$ is incomplete by Proposition \ref{prop:knk}, and hence $\{a,b,c\}$ is incomplete by Lemma \ref{lem:rprime}.

Suppose that $a+b$ does not divide $c$.
If $\gcd(a,b)=1$ and $c\leq 4$, then $\{a,b,c\}\in\{\{1,2,4\},\{2,3,4\}\}$, and  if $\gcd(a,b)=2$ and $c\leq 6$, then $\{a,b,c\}=\{2,4,6\}$. See Table \ref{tab:counterex} for deficient sequences for each of these partitions.

Suppose $\gcd(a,b)=1$ and $c\geq 5$.
let $x$ and $y$ be distinct integers such that $ax +by = c$, $-b \leq x \leq (c-1)/2$, and $-a \leq y \leq (c-1)/2$, as guaranteed by Lemma \ref{lem:ntresult}.
Without loss of generality, assume $y > x$.
Let 
\begin{equation*}
s = (a+b-1+x:a)(a+b-1+y:b)(2a+2b+c-2:c).
\end{equation*}
A direct computation shows that $s\in\mathfrak{S}^2(\{a,b,c\})$, and since 
\begin{equation*}
2a+2b+c > 2a + 2b +c - 1 = \max\{2a+2b+2y,2a+2b+c-1\},
\end{equation*}
we have that $s$ is deficient by Lemma \ref{lem:handyresult}, with $d=a+b-1+y$.

Suppose $\gcd(a,b)=2$, $c$ is odd, and $c\geq 7$.
Let $x$ and $y$ be distinct integers such that $ax +by = 2c$, $-b \leq x \leq (c-1)/2$, and $-a \leq y \leq (c-1)/2$, as guaranteed by Lemma \ref{lem:ntresult}.
Again, without loss of generality, assume $y > x$.
Let 
\begin{equation*}
s = (a+b-1+x:a)(a+b-1+y:b)(2a+2b+c-3:c).
\end{equation*}
A direct computation shows that $s\in\mathfrak{S}^2(\{a,b,c\})$, and since
\begin{equation*}
2a+2b+c > 2a+2b+c-2 \geq \max\{2a+2b+2y,2a+2b+c-2\},
\end{equation*}
we have that $s$ is deficient by Lemma \ref{lem:handyresult}, with $d=a+b-1+y$.

So $a+b$ must divide $c$.
Therefore, by Proposition \ref{prop:abaplusb}, $\{a,b,c\}=\{1,2,3\}$.
\end{proof}

Let $\Gamma_3$ be the set of partitions 
$\Gamma_3 = \{ \{1^a,2^b,3\}: a\geq 1,\ b\geq 1\}$.
We may then summarize this section with the following theorem.

\begin{theorem}
\label{thm:3ormore}
Let $\lambda$ be a partition of a positive integer $n$.
Suppose $\lambda$ has at least three part distinct  sizes and $\lambda$ is complete.
Then $\lambda\in \Gamma_3$.
\end{theorem}

\begin{proof}
Let $a$, $b$, and $c$ be distinct parts in $\lambda$.
By Lemma \ref{lem:rprime}, $\{a,b,c\}$ is complete, and hence $\{a,b,c\}=\{1,2,3\}$ by Proposition \ref{prop:abc}, and hence the only distinct part sizes in $\lambda$ are $1$, $2$, and $3$.
In addition, since the partition $3^2$ is incomplete, it follows that $\lambda$ cannot contain multiple instances of $3$.
\end{proof}

\subsection{Conjectured classification of complete partitions}

Let $\Gamma = \Gamma_1\cup \Gamma_2\cup \Gamma_3$.
Then we may combine the results of Theorems \ref{thm:12} and \ref{thm:3ormore} with the following.

\begin{theorem}
\label{thm:main}
If $\lambda$ is a complete partition, then $\lambda\in \Gamma$.
\end{theorem}

For positive integers $n$, let $\mathcal{P}_n$ be the set of partitions 
\begin{equation*}
\mathcal{P}_n = \{ \{1^a,2^b,3^\epsilon\}: a\geq 0,\ b\geq 0,\ \epsilon\in\{0,1\}\}.
\end{equation*}
Values of $|\mathfrak{C}(\mathcal{P}_n)|$ and $|\mathfrak{C}(\lambda)|$ for various partitions $\lambda\in \Gamma_2$ are given in Table~\ref{tab:clambda}, and 
in all cases presented in the table, through exhaustive search, we have that $|\mathfrak{S}^2(\mathcal{P}_n)|=|\mathfrak{C}(\mathcal{P}_n)|$ and $|\mathfrak{S}^2(\lambda)|=|\mathfrak{C}(\lambda)|$ for each $\lambda\in\Gamma_2$.
Observe that if $|\mathfrak{S}^2(\mathcal{P}_n)|=|\mathfrak{C}(\mathcal{P}_n)|$, then $|\mathfrak{S}^2(\lambda)|=|\mathfrak{C}(\lambda)|$ for each $\lambda\in \mathcal{P}_n$.
Hence, this numerical evidence leads to the following conjecture which, in turn, implies Conjectures \ref{conj:bebeacua} and \ref{conj:D4Conj}.

\begin{conjecture}
\label{conj:main}
A partition $\lambda$ is complete if and only if $\lambda\in \Gamma$.
\end{conjecture}

\begin{table}
\centering
$\begin{array}[t]{|>{\rule[-5pt]{0pt}{16pt}}c|c|c|c|}
\hline
a & |\mathfrak{C}(\{1^a,4\})| & |\mathfrak{C}(\{1^a,5\})| & |\mathfrak{C}(\{2^a,4\})| \\\hline
1 & 2 & 2 & 4\\\hline
2 & 11 & 14 & 30\\\hline
3 & 52 & 72 & 232\\\hline
4 & 228 & 339 & 1787\\\hline
5 & 932 & 1518 & 13910 \\\hline
6 & 3712 & 6551 \\\cline{1-3}
7 & 14630 & 27480 \\\cline{1-3}
\end{array}\ 
\begin{array}[t]{|>{\rule[-5pt]{0pt}{16pt}}c|c|}
\hline
n & |\mathfrak{C}^2(\mathcal{P}_n)| \\\hline
1 & 1\\\hline
2 & 2\\\hline
3 & 5\\\hline
4 & 15\\\hline
5 & 56\\\hline
6 & 226\\\hline
7 & 992\\\hline
\end{array}\ 
\begin{array}[t]{|>{\rule[-5pt]{0pt}{16pt}}c|c|}
\hline
n & |\mathfrak{C}^2(\mathcal{P}_n)| \\\hline
8 & 4553\\\hline
9 & 21708\\\hline
10 & 106363\\\hline
11 & 533554\\\hline
12 & 2726987\\\hline
13 & 14157710\\\hline
\end{array}$
\caption{A table of values for $|\mathfrak{C}_n(\lambda)|$ for various $\lambda\in\Gamma_2$, as well as $|\mathfrak{C}_n(\mathcal{P}_n)|$ for $n\leq 13$.
For each $\lambda$ presented in the table, $|\mathfrak{C}(\lambda)| = |\mathfrak{S}^2(\lambda)|$, and for each $n\leq 13$, $|\mathfrak{C}_n(\mathcal{P}_n)|=|\mathfrak{S}_n^2(\mathcal{P}_n)|$.}
\label{tab:clambda}
\end{table}

We close by mentioning that a new method of proof will be needed for these conjectures.
Consider the following proposition.

\begin{proposition}
\label{prop:nodef1n}
For all positive integers $n$, $\mathfrak{S}^2(1^n)$ does not contain any deficient sequences.
\end{proposition}

\begin{proof}
Let $n$ be a positive integer.
Assume there exists a deficient sequence $s\in \mathfrak{S}^2(1^n)$ for some $d\in Z_n$.
Let $M$ the maximal term in $s$ and $\sigma = |\{ i: s_i \leq d \}|$.
Then $n + \sigma \geq 2d + 3$ and $n + \sigma \geq M + 2$.
So $d \leq (n + \sigma - 3)/2$ and $M \leq n + \sigma - 2$.
Additionally, $\sigma \leq n$.
Therefore, 
\begin{align*}
\sum_{i=0}^{n-1} s_i 
&\leq d + \cdots + (d - (\sigma -1)) + M + \cdots + (M - (n - \sigma - 1))
\\&= \sigma d - \frac{\sigma(\sigma-1)}2+ M(n-\sigma) - \frac{(n-\sigma)(n-\sigma-1)}2
\\&\leq \frac{\sigma (n + \sigma -3)}2 - \frac{\sigma(\sigma-1)}2 + (n + \sigma - 2)(n-\sigma) - \frac{(n-\sigma)(n-\sigma-1)}2
\\&= -\frac18\left(\left(2n-3\sigma\right)^2 + 3\sigma^2 + 4(n - \sigma) \right) + n(n-1)
< n(n-1),
\end{align*}
which is a contradiction.
So $\mathfrak{S}^2(1^n)$ does not contain any deficient sequences.
\end{proof}
Using an argument similar to that which is given in the proof of Proposition \ref{prop:nodef1n}, one can show that for all $\lambda\in \Gamma$, $\mathfrak{S}^2(\lambda)$ contains no deficient sequences.

With the exception of Proposition \ref{prop:aa12a1}, every incompleteness result hinged on the existence of a deficient sequence.
In fact, the sequences presented in the proof of Proposition \ref{prop:aa12a1} demonstrate that a score sequence of a $2$-tournament can both be not deficient and not the characteristic sequence for a permutation.
However, one can show that $\mathfrak{S}^2(\{a,a+1,2a+1)$ with $a\geq 2$ and $\mathfrak{S}^2(\{a,a+2,2a+2)$ with $a\geq 1$ contain deficient sequences. 
For example, for each odd integer $a$,
\begin{equation*}
((5a-1)/2:a)((9a+1)/2:a+1)((9a+3)/2:2a+1)
\end{equation*}
is a deficient sequence in $\mathfrak{S}^2(\{a,a+1,2a+1\})$, and similar sequences can be found for the other cases arising in Proposition \ref{prop:aa12a1}.
However, a partition $\lambda$ can be incomplete while $\mathfrak{S}^2(\lambda)$ has no deficient sequences. 
Observe that while $\{1^5,6\}$ is incomplete by Proposition \ref{prop:1k} and Lemma \ref{lem:rprime}, we have that $\mathfrak{S}^2(\{1^5,6\})$ has no deficient sequences by a similar argument as given in the proof of Proposition~\ref{prop:nodef1n}.
In fact, $\mathfrak{S}^2(\{1^a,6\})$ has no deficient sequences for any $a\geq 5$.

These observations lead us to the following summative result.

\begin{theorem}
Let $\lambda$ be a partition.
If $\lambda\in \Gamma$, then $\mathfrak{S}^2(\lambda)$ contains no deficient sequences.
If $\lambda\notin\Gamma$, then there exists $\lambda'\subseteq\lambda$ such that $\lambda'$ is deficient.
\end{theorem}

With this in mind, the following conjecture is equivalent to Conjecture \ref{conj:main}.

\begin{conjecture}
\label{conj:badsequences}
A partition $\lambda$ is complete if and only if $\mathfrak{S}^2(\lambda')$ has no deficient sequences for all $\lambda'\subseteq \lambda$.
\end{conjecture}

Therefore, if a partition in $\Gamma$ were incomplete, another method of proof would be necessary, as there are no deficient sequences in $\mathfrak{S}^2(\Gamma)$ to demonstrate otherwise.

\bibliographystyle{abbrv}
\bibliography{biblio}

\end{document}